\documentclass{amsart}
\usepackage{amsmath,amssymb,amsthm}
\usepackage[all]{xy}
\newtheorem{dfn}{Definition}[section]
\newtheorem{thm}[dfn]{Theorem}
\newtheorem{lem}[dfn]{Lemma}
\newtheorem{cor}[dfn]{Corollary}
\newtheorem{rem}[dfn]{Remark}
\newtheorem{prop}[dfn]{Proposition}
\newtheorem{ex}[dfn]{Example}
\newtheorem{conj}[dfn]{Conjecture}

\begin{document}
\title{LAGRANGIAN EMBEDDINGS OF CUBIC FOURFOLDS CONTAINING A PLANE}
\author{GENKI OUCHI}
\maketitle
\begin{abstract}
We prove that a very general cubic fourfold containing a plane can be embedded into a holomorphic symplectic eightfold as a Lagrangian submanifold. We construct the desired holomorphic symplectic eightfold as a moduli space of Bridgeland stable objects in the derived category of the twisted K3 surface corresponding to the cubic fourfold containing a plane.  
\end{abstract}
\section{INTRODUCTION}
 
\subsection{Motivation and results}
 Cubic fourfolds have been studied in the context of associated holomorphic symplectic manifolds, relations to K3 surfaces and rationality problems and so on.  For example, Beauville and Donagi \cite{BD} proved that the Fano variety $F(X)$ of lines on $X$ is a holomorphic symplectic fourfold deformation equivalent to the Hilbert scheme of two points on a K3 surface. Recently, Lehn et al \cite{LLSS} proved that if $X$ is a cubic fourfold \it{not}\rm{} containing a plane, then $X$ can be embedded into a holomorphic symplectic eightfold $Z$ as a Lagrangian submanifold. The above $Z$ is constructed by the moduli space of generalized twisted cubics on $X$ \cite{JS}, and if $X$ is Pfaffian, then Addington and Lehn \cite{AL} proved that $Z$ is deformation equivalent to the Hilbert scheme of four points on a K3 surface. However, if $X$ contains a plane,  the argument of Lehn et al is not applied.
In this paper, we proved the following theorem.

\begin{thm}
Let $X$ be a very general cubic fourfold containing a plane. Then $X$ can be embedded into a holomorohic symplectic eightfold $M$ as a Lagrangian submanifold. Moreover, $M$ is deformation equivalent to the Hilbert scheme of four points on a K3 surface.
\end{thm}

Although Lehn et al used the moduli space of twisted cubics, we use notions of derived categories and Bridgeland stability conditions in our construction of $M$. More presicely, the holomorphic symplectic eightfold $M$ is constructed as a moduli space of Bridgeland stable objects in the derived category of the twisted K3 surface  $(S,\alpha)$, which corresponds to $X$.  The twisted K3 surface $(S,\alpha)$ is constructed by Kuznetsov (\cite{Kuz10}, Section 4) in the context of his conjecture about K3 surfaces and rationality of cubic fourfolds. 

\subsection{Background}
We recall Kuznetsov's conjecture. The rationality problem of cubic fourfolds is related to K3 surfaces conjectually.  The derived category $D^b(X)$ of coherent sheaves on $X$ has the following semiorthogonal decomposition:
\begin{equation}\label{sod}
 D^b(X)= \langle \mathcal{A}_X, \mathcal{O}_X, \mathcal{O}_X(1), \mathcal{O}_X(2) \rangle. 
\end{equation}

The full triangulated subcategory $\mathcal{A}_X$ is a Calabi-Yau $2$ category i.e. the Serre functor of $\mathcal{A}_X$ is isomorphic to the shift functor $[2]$. Kuznetsov proposed the following conjecture.

\begin{conj}[\cite{Kuz10}]\label{Kuzconj}
A cubic fourfold $X$ is rational if and only if there is a K3 surface $S$ such that $\mathcal{A}_X \simeq D^b(S)$. 
\end{conj}

Hassett \cite{Has00} introduced the notion of special cubic fourfolds. Cubic fourfolds containing a plane are examples of special cubic fourfolds. Special cubic fourfolds often have associated K3 surfaces Hodge theoretically \cite{Has00}. Addington and Thomas \cite{AT} proved that Kuznetsov's and Hassett's relations between cubic fourfolds and K3 surfaces coincide generically.  The known examples of rational cubic fourfolds are Pfaffian cubic fourfolds \cite{Tr84}, \cite{Tr93} and some rational cubic fourfolds containing a plane, which are constructed in \cite{Has99}. Conjectually, very general cubic fourfolds are irrational. However, there are no known examples of irrational cubic fourfolds so far. Kuznetsov constructed the equivalences between $\mathcal{A}_X$ and the derived categories of coherent sheaves on K3 surfaces for these rational cubic fourfolds. For a general cubic fourfold $X$ containing a plane, Kuznetsov proved the following theorem more generally.

\begin{thm}[\cite{Kuz10}, Theorem 4.3]
Let $X$ be a general cubic fourfold containing a plane. Then there is a twisted K3 surface $(S,\alpha)$ such that $\mathcal{A}_X \simeq D^b(S,\alpha)$. Moreover, the Brauer class $\alpha \in \mathrm{Br}(S)$ is trivial i.e. the twisted K3 surface $(S,\alpha)$ is the usual K3 surface $S$ if and only if $X$ is Hassett's rational cubic fourfold containing a plane.
\end{thm}

We say that a general cubic fourfold $X$ containing a plane is very general when the Picard number of $S$ is equal to one. If a cubic fourfold $X$ containing a plane is very general, then $\mathcal{A}_X$ is not equivalent to derived categories of coherent sheaves on K3 surfaces (\cite{Kuz10} Proposition 4.8). So very general cubic fourfolds containing a plane are irrational conjectually.

 We recall previous works on holomorphic symplectic manifolds associated to cubic fourfolds and derived categories. Using the mutation functors associated to the semiorthogonal decomposition (\ref{sod}), we can define a projection functor $\mathrm{pr} \colon D^b(X) \to \mathcal{A}_X$. The Fano variety $F(X)$ of lines on $X$ and the holomorphic sympletic eightfold $Z$ in \cite{LLSS} are related to the projection functor $\mathrm{pr} \colon D^b(X) \to \mathcal{A}_X$. In \cite{KM}, the Fano variety $F(X)$ of lines on $X$ is regarded as a moduli space of objects in $\mathcal{A}_X$ of the form $\mathrm{pr}(\mathcal{O}_{\mathrm{line}}(1))$. For a general cubic fourfold $X$ containing a plane, Macri and Stellari \cite{MS} constructed Bridgeland stability conditions on $\mathcal{A}_X \simeq D^b(S,\alpha)$ such that all objects of the form $\mathrm{pr}(\mathcal{O}_{\mathrm{line}}(1))$ are stable. So the Fano variety $F(X)$ of lines on a general cubic fourfold $X$ containing a plane is isomorphic to a moduli space of Bridgeland stable objects in $\mathcal{A}_X \simeq D^b(S,\alpha)$. For a general Pfaffian cubic fourfold $X$ not containing a plane, Lehn and Addington \cite{AL} proved that the holmorphic symplectic eightfold $Z$ is birational to the Hilbert scheme of four points on the K3 surface considering the projections of ideal sheaves of (generalized) twisted cubics on $X$ and the equivalence between $\mathcal{A}_X$ and the derived category of coherent sheaves on the K3 surface. In particular, the holomorphic symplectic eightfold $Z$ is deformation equivalent to the Hilbert scheme of four points for a general Paffian cubic fourfold not containing a plane.

\subsection{Strategy for Theorem 1.1}
To construct Lagrangian embeddingsof cubic fourfolds, we consider the projections of skyscraper sheaves of points on $X$.
 First, we illustrate the relation between the projection functor $\mathrm{pr} \colon D^b(X) \to \mathcal{A}_X$ and Lagrangian embeddings of cubic fourfolds. We prove the following proposition in Section 4.

\begin{prop}\label{Lagproj}
Let $X$ be a cubic fourfold. Take a point $x \in X$. Then the followings hold.
\begin{itemize}
\item For $x \neq y \in X$, $\mathrm{pr}(\mathcal{O}_x)$ is not isomorphic to $\mathrm{pr}(\mathcal{O}_y)$.
\item We have $\mathrm{Ext}^1(\mathcal{O}_x, \mathcal{O}_x)=\mathbb{C}^4$, $\mathrm{Ext}^1(\mathrm{pr}(\mathcal{O}_x), \mathrm{pr}(\mathcal{O}_x))=\mathbb{C}^8$ and

 $\mathrm{Ext}^2(\mathrm{pr}(\mathcal{O}_x),\mathrm{pr}(\mathcal{O}_x)) \simeq \mathrm{Hom}(\mathrm{pr}(\mathcal{O}_x),\mathrm{pr}(\mathcal{O}_x))=\mathbb{C}$. 
\item The linear map $\mathrm{pr} \colon \mathrm{Ext}^1(\mathcal{O}_x, \mathcal{O}_x) \to \mathrm{Ext}^1(\mathrm{pr}(\mathcal{O}_x),\mathrm{pr}(\mathcal{O}_x))$ is injective.
\item Let 
\[ \omega_x \colon \mathrm{Ext}^1(\mathrm{pr}(\mathcal{O}_x),\mathrm{pr}(\mathcal{O}_x)) \times \mathrm{Ext}^1(\mathrm{pr}(\mathcal{O}_x),\mathrm{pr}(\mathcal{O}_x)) \to \mathrm{Ext}^2(\mathrm{pr}(\mathcal{O}_x),\mathrm{pr}(\mathcal{O}_x))\]
 be the bilinear form induced by the composition of morphisms in the derived category. Then the bilinear form $\omega_x$ vanishes on $\mathrm{Ext}^1(\mathcal{O}_x, \mathcal{O}_x)$.
\end{itemize} 
\end{prop} 

Next, we construct a Lagrangian embedding of a very general cubic fourfold containing a plane using Bridgeland stability conditions $\sigma$ on the Calabi-Yau $2$ category $\mathcal{A}_X$ such that the objects $\mathrm{pr}(\mathcal{O}_x)$ are $\sigma$-stable for all $x \in X$. We prove the following proposition.

\begin{prop}[Proposition \ref{main}]\label{mainthm}
Let $X$ be a very general cubic fourfold containing a plane and $\Phi \colon \mathcal{A}_X \stackrel{\sim}{\to} D^b(S,\alpha)$ be the equivalence as in Corollary 2.14. Let $v$ be the Mukai vector of  $\Phi(\mathrm{pr}(\mathcal{O}_x))$.  Then there is a stability condition $\sigma \in \mathrm{Stab}(D^b(S,\alpha))$ generic with respect to $v$ such that $\mathrm{pr}(\mathcal{O}_x)$ is $\sigma$-stable for all $x \in X$. In particular, the morphism
\[ X \to M, x \mapsto \Phi(\mathrm{pr}(\mathcal{O}_x)) \] is the Lagrangian embedding. Here M is the moduli space of $\sigma$-stable objects with Mukai vector $v$.  So $M$ is deformation equivalent to the Hilbert scheme of four points on a K3 surface.
\end{prop}

In Proposition 1.4, we don't assume that a cubic fourfold $X$ doesn't contain a plane. However, we assume that $X$ is a very general cubic fourfold containing a plane in Proposition \ref{mainthm}. Since we don't know how to construct stability conditions on $\mathcal{A}_X$ for a general cubic fourfold $X$ so far, we need to use some geometric discription of $\mathcal{A}_X$ in order to construct Bridgeland stability conditions on $\mathcal{A}_X$. In fact, it is difficult to construct the heart $\mathcal{C}$ of a bounded t-structure on $\mathcal{A}_X$ and a central charge $Z \colon K(\mathcal{A}_X) \to \mathbb{C}$ such that $Z(\mathcal{C} \setminus \{0\})$ is contained in the semiclosed upper-half plane.   Moreover, we don't have well-established moduli theory for Bridgeland stable objects in $\mathcal{A}_X$. So we need some (twisted) K3 surfaces to use moduli theory for Bridgeland stable objects as in \cite{BM12}, \cite{BM13}. However, if $X$ is a very general cubic fourfold containing a plane, we can construct desired Bridgeland stability conditions on $\mathcal{A}_X$  using the twisted K3 surface $(S,\alpha)$. Thus, using the moduli theory \cite{BM12}, \cite{BM13} of Bridgeland stable objects on derived categories of twisted K3 surfaces, we have the Lagrangian embedding $X \to M$ in Proposition 1.5. So we obtain Theorem 1.1.

 Finally we coments on two recent works on cubic fourfolds.  One is the work on Bridgeland stability conditions on $\mathcal{A}_X$ by Toda \cite{Tod13}. By the Orlov's theorem \cite{Orl09}, the triangulated category $\mathcal{A}_X$ is equivalent to the triangulated category $\mathrm{HMF}^{gr}(W)$ of graded matrix factorizations of the defining polynomial $W$ of $X$. To investigate Bridgeland stability conditions on $\mathcal{A}_X$ is related to the existence problem of Gepner type stability condition on $\mathrm{HMF}^{gr}(W)$, which is treated in \cite{Tod13}. However, it is also difficult to construct the heart of a bounded t-structure on $\mathrm{HMF}^{gr}(W)$ so far. Other one is the work on rationality problem of cubic fourfolds and Fano variety of lines by Galkin and Shinder \cite{GS}. Galkin and Shinder \cite{GS} proved that rationality of cubic fourfolds is related to birationality of Fano varieties of lines and Hilbert schemes of two points on K3 surfaces if Cancellation conjecture on the Grothendieck ring of varieties holds. Addington \cite{Ad} compared results in \cite{GS} with Conjecture \ref{Kuzconj}. It may be interesting to study relationship between Lagrangian embeddings of cubic fourfolds and rationality of cubic fourfolds.

\subsection*{Construction of this paper}
In Section 2, we recall the notion of Bridgeland stability conditions on derived categories of twisted K3 surfaces (\cite{Bri07},\cite{Bri08},\cite{HMS}), properties of moduli spaces of Bridgeland stable objects in derived categories of twisted K3 surfaces (\cite{BM12},\cite{BM13}), and the construction of Kuznetsov equivalence (\cite{Kuz10}).
In Section 3, we define the projection functor $\mathrm{pr} \colon D^b(X) \to \mathcal{A}_X$ and explain Proposition \ref{mainthm} more precisely. In Section 4, we see properties of the projection functor and prove Proposition \ref{Lagproj}. In Section 5, we calculate the images of objects in $\mathcal{A}_X$ via Kuznetsov equivalence and compute their Mukai vectors.
In Section 6, we construct Bridgeland stability conditions on derived category of the twisted K3 surface such that projections of structure sheaves of points in $X$ are stable. In this section, we complete the proof of Proposition \ref{mainthm}

\subsection*{Notation}
We work over the complex number field $\mathbb{C}$. Cubic fourfolds and K3 surfaces are always smooth and projective. A triangulated category means a $\mathbb{C}$-linear triangulated category. For a smooth projetive variety $X$, we denote by $D^b(X)$ the bounded derived category of coherent sheaves on $X$. We write its Grothendieck group as $K(X):=K(D^b(X))$. For an object $E \in \mathcal{D}$ in a Calabi-Yau 2 category $\mathcal{D}$, we say that $E$ is spherical if $\mathbf{R}\mathrm{Hom}(E,E)=\mathbb{C} \oplus \mathbb{C}[-2]$.

Let $\mathcal{D}$ be a triangulated category. For an exceptional object $E \in \mathcal{D}$, we define the right mutation functor $\mathbf{R}_E \colon \mathcal{D} \to \mathcal{D}$ and the left mutation functor $\mathbf{L}_E \colon \mathcal{D} \to \mathcal{D}$ as follows
\begin{align*}
\mathbf{R}_E(-)&:=\mathrm{Cone}(- \to \mathbf{R}\mathrm{Hom}(-,E)^{\vee} \otimes E)[-1]\\
\mathbf{L}_E(-)&:=\mathrm{Cone}(\mathbf{R}\mathrm{Hom}(E,-) \otimes E \to -).
\end{align*}

\subsection*{Acknowledgements}
I would like to express my sincere gratitude to my advisor Professor Yukinobu Toda for his valuable comments and warmful encouragement.  I would like to thank Professor Hokuto Uehara, Professor Shinnosuke Okawa and Professor Daisuke Matsushita. They gave me the chances to talk about my result in this paper at DMM seminar at Kavli IPMU, the workshop at Osaka University and the workshop at RIMS in Kyoto respectively.  This work was supported by the program for Leading Graduate Schools, MEXT, Japan.

\section{PRELIMINARY}
In this section, we recall the notions of twisted K3 surfaces and Bridgeland stability conditions, and the relation between cubic fourfolds containing a plane and twisted K3 surfaces. 
\subsection{Twisted K3 surfaces} 
We review the definitions of twisted K3 surfaces, twisted sheaves and the twisted Mukai lattices.
\begin{dfn}[\cite{C}]
A twisted K3 surface is a pair $(S,\alpha)$ of a K3 surface $S$ and an element $\alpha$ of the Brauer group $\mathrm{Br}(S):=H^2(S,\mathcal{O}_{S}^*)_{\mathrm{tor}}$ of $S$.
\end{dfn}

\begin{dfn}[\cite{C}]
Let $(S,\alpha)$ be a twisted K3 surface. Taking an analytic open cover $\{U_i\}_{i \in I}$ of $S$, the Brauer class $\alpha$ can be represented by   $\check{C}ech$ cocycle $\{\alpha_{ijk}\}$. An $\alpha$-twisted coherent sheaf $F$ on $S$ is a collection $(\{F_i\}_{i \in I},\{\phi_{ij}\}_{i,j \in I})$, where $F_i$ is a coherent sheaf on $U_i$ and $\phi_{ij}|_{U_i \cap U_j }\colon F_{i}|_{U_i \cap U_j} \to F_{j}|_{U_i  \cap U_j} $ is an isomorphism satisfying the following conditions:
\[\phi_{ii}=\mathrm{id}, \ \phi_{ij}=\phi_{ji}^{-1}, \ \phi_{ij}\circ\phi_{jk}\circ\phi_{ki}=\alpha_{ijk}\cdot\mathrm{id}.\]
We denote by $\mathrm{Coh}(S,\alpha)$ and set $D^b(S,\alpha):=D^b(\mathrm{Coh}(S,\alpha))$ the category of $\alpha$-twisted coherent sheaves on $S$.
\end{dfn}

Let $(S,\alpha)$ be a twisted K3 surface. For simplicity, we will say $E \in \mathrm{Coh}(S,\alpha)$ a sheaf instead of an $\alpha$-twisted sheaf. 

Take $B \in H^2(S,\mathbb{Q})$ with $\exp{(B^{0,2})}=\alpha$. Then $B$ is called a $B$-field of $\alpha$.
Here $B^{0,2}$ is the $(0,2)$-part of $B$ in $H^2(S,\mathbb{C})$. We define the twisted Mukai lattice $\widetilde{H}^{1,1}(S,B,\mathbb{Z})$  as follow:
\[ \widetilde{H}^{1,1}(S,B,\mathbb{Z}):=e^B\biggl(\bigoplus_{i=0}^2 H^{i,i}(S,\mathbb{Q})\biggr) \cap H^{*}(S,\mathbb{Z}).\]
The lattice structure is given by the Mukai pairing $\langle-,-\rangle$:
\[ \langle (r,c,d),(r^{\prime},c^{\prime},d^{\prime}) \rangle :=cc^{\prime}-rd^{\prime}-dr^{\prime}.\]

There is the twisted Chern character \cite{HS} 
\[ \mathrm{ch}^{B}:K(S,\alpha) \twoheadrightarrow \widetilde{H}^{1,1}(S,B,\mathbb{Z}).\]
The twisted Chern character $\mathrm{ch}^B$ satisfies the Riemann-Roch formula: 
\begin{equation}\label{RR}
 \chi(E,F)=-\langle v^B(E),v^B(F) \rangle. 
\end{equation}  
Here $v^B(E):=\mathrm{ch}^{B}(E) \cdot \sqrt{\mathrm{td_{S}}} \in \widetilde{H}^{1,1}(S,B,\mathbb{Z})$ is the (twisted) Mukai vector of $E \in K(S,\alpha)$. We denote by $c_{1}^B(-)$ the degree 2 part of $v^B(-)$.

\begin{rem}
Let $N(S,\alpha)$ be the numerical Grothendieck group of $D^b(S,\alpha)$. The twisted Chern character induces the isomorphism
\[ \mathrm{ch}^B \colon N(S,\alpha) \to \widetilde{H}^{1,1}(S,B,\mathbb{Z}).\]
\end{rem}
\begin{lem}[\cite{MS}, Lemma 3.1]\label{gene}
Let $d$ be the order of $\alpha$. Then the twisted Mukai lattice $\widetilde{H}^{1,1}(S,B,\mathbb{Z})$ is generated by $(d,dB,0), \mathrm{Pic}(S)$ and $(0,0,1)$ in $H^{*}(S,\mathbb{Z})$.
In particular, the rank of $E$ is divisible by $d$ for all $E \in D^b(S,\alpha)$.
\end{lem}

\subsection{Bridgeland stability conditions}
Let $\mathcal{D}$ be a triangulated category and $N(\mathcal{D})$ be the numerical Grothendieck group of $\mathcal{D}$. Assume that $N(\mathcal{D})$ is finitely generated. If $\mathcal{D}$ is the derived category of a twisted K3 surface, this assumption is satisfied.

\begin{dfn}[\cite{Bri07}]
A stability condition on $\mathcal{D}$ is a pair $\sigma = (Z,\mathcal{C})$ of a group homomorphism (called central charge) $Z:N(\mathcal{D}) \to \mathbb{C}$  and the heart of a bounded t-structure $\mathcal{C \subset \mathcal{D}}$ on $\mathcal{D}$, which satisfy the following conditions:
\begin{itemize}
\item For any $0 \neq E \in \mathcal{C}$, we have $Z(E) \in \{re^{i\pi \phi} \in \mathbb{C} \mid r>0,0<\phi \le 1\}.$ 
\item For any $0 \neq E \in \mathcal{C}$, there is a filtration (called Harder-Narasimhan filtration) in $\mathcal{C}$
\[0=E_0 \subset E_1 \subset \cdot \cdot \cdot \subset E_N =E\]
such that $F_i := E_i/E_{i-1}$ is $\sigma$-semistable and $\phi(F_i)>\phi(F_{i+1})$ for all $1\le i \le N-1$.
\item Fix a norm $||-||$ on $N(\mathcal{D})_{\mathbb{R}}$. Then there is a constant $C$ such that $||E|| \le C \cdot |Z(E)|$ for any non-zero $\sigma$-semistable object $E \in \mathcal{C}$. This property is called the support property.
\end{itemize}
Here we put $\phi(E):=\mathrm{arg}(Z(E))/\pi \in (0,1]$ for $0 \neq E \in \mathcal{C}$ and $E \in \mathcal{C}$ is $\sigma$-(semi)stable if the inequality $\phi(F)<(\le)\phi(E)$ holds for any $0 \neq F \subset E$.
\end{dfn}

\begin{rem}[\cite{Bri07}]
We denote  by $\mathrm{Stab}(\mathcal{D})$ the set of all stability conditions on $\mathcal{D}$. Then $\mathrm{Stab}(\mathcal{D})$ has a natural topology such that the map 
\[\mathrm{Stab} (\mathcal{D}) \to \mathrm{Hom}_{\mathbb{Z}}(N(\mathcal{D}) ,\mathbb{C}), (Z,\mathcal{C}) \mapsto Z \]
is a local homeomorphism. In particular, $\mathrm{Stab}(\mathcal{D})$ has a structure of a complex manifold. 
\end{rem}

From now on, we focus on stability conditions on derived categories of twisted K3 surfaces. Let $(S,\alpha)$ be a twisted K3 surface and fix a $B$-field $B \in H^2(S,\mathbb{Q})$ of the Brauer class $\alpha$.
We set $\mathrm{Stab}(S,\alpha):=\mathrm{Stab}(D^b(S,\alpha))$.

\begin{dfn}
Fix an ample divisor $\omega \in \mathrm{NS}(S)$ on $S$. Let $E \in \mathrm{Coh}(S,\alpha)$ be a sheaf. We define the slope $\mu^B(E)$ of $E$ as follow:
\[ \mu^B(E):=\frac{c_1^B(E) \cdot \omega}{\mathrm{rk}E}. \] 
If $\mathrm{rk}E=0$, then we regard $\mu^B(E)=\infty$.
We say that $E$ is $\mu^B$-(semi)stable if and only if $\mu^B(F)(\le)<\mu^B(E/F)$ holds for all nonzero subsheaves $F \subset E$.
\end{dfn}

Note that the $\mu^B$-satability admits the Harder-Narasimhan filtrations and Jordan-H\"{o}lder filtrations.

\begin{ex}[\cite{Bri08}, \cite{HS}]\label{stab}
Take $B^{\prime} \in \mathrm{NS}(S)_{\mathbb{R}}$ and a real ample class $\omega \in \mathrm{NS}(S)_{\mathbb{R}}$ with $\omega^2>2$. Let $\tilde{B}:=B^{\prime}+B \in H^2(S,\mathbb{R})$. We define a group homomorphism $Z:=Z_{\tilde{B},\omega} \colon N(S,\alpha) \to \mathbb{C}$ as follow:
\[Z_{\tilde{B},\omega}(E):=\langle v^B(E),e^{\tilde{B}+i\omega} \rangle. \]
We can define a torsion pair $(\mathcal{T},{F})$ on $\mathrm{Coh}(S,\alpha)$ as follows:
\begin{itemize}
\item $\mathcal{T}:= \langle E \in \mathrm{Coh}(S,\alpha) \mid$ $E$ is $\mu^B$-semistable with $\mu^B(E)>\tilde{B}\omega \rangle_{\mathrm{ex}}$
\item $\mathcal{F}:= \langle E \in \mathrm{Coh}(S,\alpha) \mid E$ is $\mu^B$-semistable with $\mu(E) \le \tilde{B}\omega \rangle_{\mathrm{ex}}$.
\end{itemize}
Then $\mathcal{C}:=\langle \mathcal{F}[1],\mathcal{T} \rangle_{\mathrm{ex}} \subset D^b(S,\alpha)$ is the heart of a bounded t-structure on $D^b(S,\alpha)$ induced by the torsion pair $(\mathcal{T},\mathcal{F})$. Here we denote the extension closure by $\langle - \rangle_{\mathrm{ex}}$. The pair $(Z,\mathcal{C})$ is a stability condition on $D^b(S,\alpha)$.
\end{ex}

Let $\mathrm{Stab}^{\dagger}(S,\alpha)$ be the conected component of the space of stability conditions $\mathrm{Stab}(S,\alpha)$, which contains the stability conditions of the form $(Z_{\tilde{B},\omega},\mathcal{C})$.

\begin{rem}[\cite{Bri08}, \cite{Tod08}, \cite{BM13}]
Fix a Mukai vector $v \in  \widetilde{H}^{1,1}(S,B,\mathbb{Z})$. Then $\mathrm{Stab}^{\dagger}(S,\alpha)$ has a wall and chamber structure which depends only on a choice of $v$.  Variying $\sigma \in \mathrm{Stab}^{\dagger}(S,\alpha)$ within a chamber,  the set of $\sigma$-(semi)stable objects with Mukai vector $v$ does not change.  If $\sigma \in \mathrm{Stab}^{\dagger}(S,\alpha)$ is in a chamber, we say $\sigma$ is generic with respect to $v$ . If $v$ is primitive, then $\sigma \in   \mathrm{Stab}^{\dagger}(S,\alpha)$ is generic with respect to $v$ if and only if all $\sigma$-semistable objects with Mukai vector $v$ are $\sigma$-stable. 
\end{rem}

\subsection{Moduli spaces of Bridgeland stable complexes on twisted K3 surfaces}
We recall the facts on moduli spaces of Bridgeland stable objects on twisted K3 surfaces.

\begin{dfn}
A holomorphic symplectic variety is a simply connected smooth projective variety $M$ with a non-degenerate holomorphic 2-form $\omega$ (called symplectic form) such that $H^0(M,\Omega_{M}^2)=\mathbb{C}\cdot \omega$.
\end{dfn}

Examples of holomorphic symplectic varieties which will be appeared later are moduli spaces of Bridgeland stable objects in derived categories of twisted K3 surfaces. 

\begin{thm}[\cite{BM13}]\label{moduli}
Let $(S,\alpha)$ be a twisted K3 surface and $v \in  \widetilde{H}^{1,1}(S,B,\mathbb{Z})$ be a primitive Mukai vector with $\langle v,v \rangle \ge -2$. Let $\sigma \in \mathrm{Stab}^{\dagger}(S,\alpha)$ be a stability condition generic with respect to $v$.
Then the coarse moduli space $M_{\sigma}(v)$  of $\sigma$-stable objects with Mukai vector $v$ is a holomorphic symplectic variety deformation equivalent to the Hilbert scheme of points of a K3 surface and $\dim{M_{\sigma}(v)}=2+\langle v,v \rangle$.
\end{thm}

\subsection{Relation between cubic fourfolds and twisted K3 surfaces}
Let $X$ be a cubic fourfold and $H$ be a hyperplane section of $X$.
Consider the semiorthogonal decomposition: 
\[ D^b(X)=\langle \mathcal{A}_X ,\mathcal{O}_X ,\mathcal{O}_X(H),\mathcal{O}_X(2H) \rangle.\]
The full triangulated subcategory
\[ \mathcal{A}_X=\{E \in D^b(X) \mid \mathbf{R}\mathrm{Hom}(\mathcal{O}_X(iH),E)=0 ,i=0,1,2 \} \subset D^b(X) \]
is a Calabi-Yau 2 category(\cite{Kuz03}, Corollary 4.3).

We recall geometric properties of cubic fourfolds containing a plane \cite{Has99}, \cite{Kuz10}. Suppose that $X$ contains a plane $P=\mathbb{P}^2$ in $\mathbb{P}^5$.
Let $\sigma\colon \tilde{X} \to X$ be the blowing up of $X$ at the plane P and $p\colon \widetilde{\mathbb{P}^5} \to \mathbb{P}^5$ be the blowing up of $\mathbb{P}^5$ at the plane $P$.  The linear projection from $P$ gives the morphism $q\colon \widetilde{\mathbb{P}^5} \to \mathbb{P}^2$. This is a projectivization of the rank 4 vector bundle $\mathcal{O}_{\mathbb{P}^2}^{\oplus 3} \oplus \mathcal{O}_{\mathbb{P}^2}(-h)$ on $\mathbb{P}^2$. Here $h$ is a line in $\mathbb{P}^2$. Let $D$ be the exceptional divisor of $\sigma$. Then $D$ is linearly equivalent to $H-h$ on $\tilde{X}$.  Set $\pi:=q\circ j\colon \tilde{X} \to \mathbb{P}^2$, where $j\colon \tilde{X} \hookrightarrow \widetilde{\mathbb{P}^5}$ is the natural inclusion. Then $\pi \colon \tilde{X} \to \mathbb{P}^2$ is a quadric fibration with degenerate fibres along a plane curve $C$ of degree $6$. We assume that fibres of $\pi$ don't degenerate into union of two planes.
Then $C$ is a  smooth curve. Let $f\colon S \to \mathbb{P}^2$ be the double cover ramified along $C$.
Since $C$ is smooth, the surface $S$ is a K3 surface.

\[ \xymatrix{D \ar@{^{(}-{>}} [r] \ar[d] &\tilde{X} \ar@{^{(}-{>}} [r]^j \ar[d]_\sigma& \widetilde{\mathbb{P}^5} \ar[d]_p \ar[rd]_q && \\ P \ar@{^{(}-{>}} [r] & X \ar@{^{(}-{>}} [r] & \mathbb{P}^5 \ar@{.{>}} [r]& \mathbb{P}^2  & \ar[l]_f S} \]

We recall Kuznetsov's construction \cite{Kuz10} of the twisted K3 surface $(S,\alpha)$ and equivalence between $\mathcal{A}_X$ and $D^b(S,\alpha)$.

The quadric fibration $\pi$ defines the sheaf of Clifford algebras $\mathcal{B}$ on $\mathbb{P}^2$. It has the even part $\mathcal{B}_{0}$ and the odd part $\mathcal{B}_{1}$, which are described as
\[ \mathcal{B}_{0}=\mathcal{O}_{\mathbb{P}^2}\oplus\mathcal{O}_{\mathbb{P}^2}(-h)^{\oplus3}\oplus\mathcal{O}_{\mathbb{P}^2}(-2h)^{\oplus3}\oplus\mathcal{O}_{\mathbb{P}^2}(-3h) \]
\[ \mathcal{B}_{1}=\mathcal{O}_{\mathbb{P}^2}^{\oplus3}\oplus\mathcal{O}_{\mathbb{P}^2}(-h)^{\oplus2}\oplus\mathcal{O}_{\mathbb{P}^2}(-2h)^{\oplus3}. \]
Let $\mathrm{Coh}(\mathbb{P}^2,\mathcal{B}_0)$ be the category of coherent right $\mathcal{B}_0$-modules on $\mathbb{P}^2$.
Note that $\mathcal{B}_0$ is a spherical object in $D^b(\mathbb{P}^2,\mathcal{B}_0)$. Set $D^b(\mathbb{P}^2, \mathcal{B}_0) := D^b(\mathrm{Coh}(\mathbb{P}^2, \mathcal{B}_0))$.

\begin{lem}[\cite{Kuz08}, \cite{Kuz10}]
There exists a fully faithful functor
 \[ \Phi\colon D^b(\mathbb{P}^2,\mathcal{B}_0) \hookrightarrow D^b(\tilde{X}) \]
with the semiorthgonal decomposition 
\[ D^b(\tilde{X})=\langle \Phi(D^b(\mathbb{P}^2,\mathcal{B}_0)),\pi^*D^b(\mathbb{P}^2),\pi^*D^b(\mathbb{P}^2)(H) \rangle. \]
The left adjoint functor $\Psi\colon D^b(\tilde{X}) \to D^b(\mathbb{P}^2,\mathcal{B}_0)$ of $\Phi$ is described as 
\[ \Psi(-)=\mathbf{R}\pi_*((-)\otimes\mathcal{O}_{\tilde{X}}(h)\otimes\mathcal{E})[2]. \]
Here $\mathcal{E}$ is the rank 4 vector bundle on $\tilde{X}$ with a structure of a flat right $\pi^*\mathcal{B}_0$-module and the exact sequence
\begin{equation}\label{surj}
 0 \to q^*\mathcal{B}_1(-h-2H) \to q^*\mathcal{B}_0(-H) \to j_*\mathcal{E} \to 0 .
\end{equation}
\end{lem}

\begin{lem}[\cite{Kuz10}]
The followings hold.
\itemize
\item The functor
\[ \Phi_{\mathbb{P}^2}:=\mathbf{R}\sigma_*\mathbf{L}_{\mathcal{O}_{\tilde{X}}(h-H)}\mathbf{R}_{\mathcal{O}_{\tilde{X}}(-h)}\Phi\colon D^b(\mathbb{P}^2,\mathcal{B}_0) \to \mathcal{A}_X \]
gives an equivalence.
\item There is a sheaf $\mathcal{B}$ of Azumaya algebras on $S$ such that $f_*\mathcal{B}=\mathcal{B}_0$ and $f_*\colon \mathrm{Coh}(S,\mathcal{B}) \to \mathrm{Coh}(\mathbb{P}^2,\mathcal{B}_0)$ gives an equivalence.
\item There are a Brauer class $\alpha$ of order $2$ and a rank $2$ vector bundle $\mathcal{U}_0 \in \mathrm{Coh}(S,\alpha)$ such that $\otimes\mathcal{U}_0^{\vee}\colon \mathrm{Coh}(S,\alpha) \to \mathrm{Coh}(S,\mathcal{B})$ gives an equivalence.
\end{lem}

\begin{cor}\label{Kuzeq}
The functor $\Phi_{S}:=\Phi_{\mathbb{P}^2} \circ f_* \circ \otimes\mathcal{U}_0^{\vee}\colon D^b(S,\alpha) \to \mathcal{A}_X$ is an equivalence.
\end{cor}

\begin{rem}
The following holds.
\[ \Phi_{\mathbb{P}^2}^{-1}=\Psi\mathbf{L}_{\mathcal{O}_{\tilde{X}}(-h)}\mathbf{R}_{\mathcal{O}_{\tilde{X}}(h-H)}\mathbf{L}\sigma^*\colon\mathcal{A}_X \to D^b(\mathbb{P}^2,\mathcal{B}_0) \]
\end{rem}

If $X$ is very general i.e. $\mathrm{Pic}S=\mathbb{Z}$, then $\alpha$ is non-trivial.

\begin{prop}[\cite{Kuz10}, Proposition 4.8]
If $X$ is very general, $\mathcal{A}_X$ is not equivalent to $D^b(S^{\prime})$  for any K3 surface $S^{\prime}$. In particular, $\alpha \neq 1$.
\end{prop}

Due to Lemma \ref{gene}, the condition $\alpha \neq 1$ is strong constraint. In fact, if $\alpha \neq 1$, then there are no rank one sheaves on $(S,\alpha)$. 

The following lemma will be needed later.
\begin{lem}[\cite{MS}, Lemma 2.4]\label{reducelem}
The followings hold.
\begin{itemize}
\item For any $m \in \mathbb{Z}$, $\Psi(\mathcal{O}_{\tilde{X}}(mh))=\Psi(\mathcal{O}_{\tilde{X}}(mh-H))=0$.
\item $\Psi(\mathcal{O}_{\tilde{X}}(-h+H))=\mathcal{B}_0[2]$, $\Psi(\mathcal{O}_{\tilde{X}}(h-2H))=\mathcal{B}_1$.
\end{itemize}
\end{lem}

In the next section, we see the construction of the Lagrangian embeddings.
\section{FORMULATION OF THE MAIN PROPOSITION}
In this section, we define the projection functor and formulate Proposition \ref{mainthm}.

\begin{dfn}
Let $X$ be a cubic fourfold and $H$ be a hyperplane section of $X$.
We define the projection functor as follow.
\begin{equation}
\mathrm{pr}:=\mathbf{R}_{\mathcal{O}_{X}(-H)}\mathbf{L}_{\mathcal{O}_X}\mathbf{L}_{\mathcal{O}_{X}(H)}[1]\colon D^b(X) \to \mathcal{A}_X
\end{equation}
\end{dfn}

From now on, we use the same notation as in Section 2.4. 
\begin{dfn}
For a point $x\in X$, let $P_x:=\Phi_{S}^{-1}(\mathrm{pr}(\mathcal{O}_x))[-4]\in D^b(S,\alpha)$.
\end{dfn}

The following proposition is the more  precise version of Proposition \ref{mainthm}. 

\begin{prop}[Proposition \ref{mainthm}]\label{main}
Assume that $X$ is a very general cubic fourfold containing a plane $P$. Fix a $B$-field $B \in H^2(S,\mathbb{Q})$ of the Brauer class $\alpha$ and let $v:=v^B(P_x) \in \widetilde{H}^{1,1}(S,B,\mathbb{Z})$ for $x \in X$. Then the follwings hold.
\begin{itemize}
\item[\rm{(a)}] There is a stability condition $\sigma \in \mathrm{Stab}^{\dagger}(S,\alpha)$ generic with respect $v$ such that $P_x$ is $\sigma$-stable for each $x \in X$.
\item[\rm{(b)}] $M_{\sigma}(v)$ is a holomorphic symplectic eightfold.
\item[\rm{(c)}] $X \to M_{\sigma}(v) , x \mapsto P_x$ is a closed immersion.
\item[\rm{(d)}] $X$ is a Lagrangian submanifold of $M_{\sigma}(v)$.
\end{itemize}
\end{prop}

In the rest of the paper, we will give a proof of Proposition \ref{main}.  In the proof of (a), we will construct a family $\{\sigma_{\lambda}\}$ of stability conditions generic with respect to $v$ such that $P_x$ is $\sigma_{\lambda}$-stable for each $x\in X$. The construction of stability conditions will be in Section 6. The statement (b) will be deduced by Theorem \ref{moduli} and  $\mathbf{R}\mathrm{Hom}(\mathrm{pr}(\mathcal{O}_x), \mathrm{pr}(\mathcal{O}_x))=\mathbb{C}\oplus\mathbb{C}^8[-1]\oplus\mathbb{C}[-2]$ or $\langle v,v \rangle=6$. The Mukai vector $v$ will be calculated in Section 5. In the proof of (c) and (d), we identfy tangent spaces $\mathrm{T}_xX$ and $\mathrm{T}_xM_{\sigma}(v)$ with $\mathrm{Ext}^1(\mathcal{O}_x,\mathcal{O}_x)$ and $\mathrm{Ext}^1(P_x,P_x)$ respectively.  The statements (c), (d) is deduced from Proposition \ref{Lagproj}. This will be in Section 4. Note that we will not use K3 surfaces and the plane $P$ in a cubic fourfold $X$ in the proof of Proposition \ref{Lagproj}.

\section{THE PROJECTION FUNCTOR AND LAGRANGIAN EMBEDDINGS}
In this section, we prove Proposition \ref{Lagproj}. Let $X$ be a cubic fourfold and $H$ be a hyperplane section of $X$. Take a point $x \in X$. Let $I_x \subset \mathcal{O}_X$ be the ideal sheaf of $x \in X$. First, we calculate the image $\mathrm{pr}(\mathcal{O}_x)$ of the skyscraper sheaf $\mathcal{O}_x$.

\begin{lem}\label{sky}
Let $L \in \mathrm{Pic}X$ be a line bundle on $X$. The followings hold.
\begin{itemize}
\item $\mathbf{R}\mathcal{H}om(\mathcal{O}_x,L)=\mathcal{O}_x[-4].$
\item $\mathbf{R}\mathrm{Hom}(\mathcal{O}_x,L)=\mathbb{C}[-4].$
\end{itemize}
\end{lem}
\begin{proof}
The second claim is deduced by the first claim. So we prove the first claim. Let $i_x \colon x \hookrightarrow X$ be the natural inclusion. Using the Grothendieck-Verdier duality, we have the isomorphisms
\begin{align*}
 \mathbf{R}\mathcal{H}om(\mathcal{O}_x,L)&=\mathbf{R}\mathcal{H}om(i_{x*}\mathcal{O}_x,L)\\
                                                       &\simeq i_{x*}\mathbf{R}\mathcal{H}om_x(\mathcal{O}_x,i_x^{!}L) \\
                                                       &\simeq i_{x*}\mathbf{R}\mathcal{H}om_x(\mathcal{O}_x,\mathcal{O}_x[-4])\\
                                                       &\simeq \mathcal{O}_x[-4]. 
\end{align*}
\end{proof}

Consider the exact sequence 
\begin{equation}\label{ideal}
 0 \to I_x(H) \to \mathcal{O}_X(H) \to \mathcal{O}_x \to 0. 
\end{equation}
Let $e_1 \colon \mathcal{O}_x \to I_x(H)[1]$ be the extension morphism of (\ref{ideal}).

Since $\mathbf{R}\mathrm{Hom}(\mathcal{O}_X(H),I_x(H))=\mathbf{R}\Gamma(X,I_x)=0$, we have  $I_x(H) \in \langle \mathcal{O}_X(H) \rangle^{\bot}$. This implies
\[ \mathbf{L}_{\mathcal{O}_X(H)}(\mathcal{O}_x)[-1]=I_x(H). \]
Since $I_x(H) \subset \mathcal{O}_X(H)$ is genarated by five linear functions on $X$,  we have the surjection
\[  \mathcal{O}_X^{\oplus5} \twoheadrightarrow I_x(H). \]
Let $F_x:=\mathrm{Ker}(\mathcal{O}_X^{\oplus5} \twoheadrightarrow I_x(H))$. 
Consider the exact seqence
\begin{equation}\label{mid}
 0 \to F_x \to \mathcal{O}_X^{\oplus5} \to I_x(H) \to 0.
\end{equation}
Let $e_2 \colon I_x(H) \to F_x[1]$ be the extension morphism of (\ref{mid}).
If $F_x \in \langle \mathcal{O}_X \rangle^{\bot}$, we can get 
\[ \mathbf{L}_{\mathcal{O}_X}(I_x(H))[-1]=F_x. \]
In fact, the following holds.

\begin{lem}
We have $\mathbf{R}\mathrm{Hom}(\mathcal{O}_X,F_x)=0.$
\end{lem}
\begin{proof}
Applying $\mathbf{R}\Gamma(X,-)$ to the exact sequence (\ref{ideal}), we have the exact triangle
\[  \mathbf{R}\Gamma(X,I_x(H)) \to \mathbf{R}\Gamma(X,\mathcal{O}_X(H)) \twoheadrightarrow \mathbf{R}\Gamma(X,\mathcal{O}_x). \]
Since $\mathbf{R}\Gamma(X,\mathcal{O}_X(H))=\mathbb{C}^6$ and $\mathbf{R}\Gamma(X,\mathcal{O}_x)=\mathbb{C}$, we have $\mathbf{R}\Gamma(X,I_x(H))=\mathbb{C}^5$. 

Applying $\mathbf{R}\Gamma(X,-)$ to the exact sequence (\ref{mid}),
we have an exact triangle\[ \mathbf{R}\Gamma(X,F_x) \to \mathbf{R}\Gamma(X,\mathcal{O}_X^{\oplus5}) \twoheadrightarrow \mathbf{R}\Gamma(X,I_x(H)). \]
Since $\mathbf{R}\Gamma(X,\mathcal{O}_X^{\oplus5})=\mathbf{R}\Gamma(X,I_x(H))=\mathbb{C}^5$, we have 
\[\mathbf{R}\mathrm{Hom}(\mathcal{O}_X,F_x)=\mathbf{R}\Gamma(X,F_x)=0.\]
\end{proof}

By the definition of the right mutation functor $\mathbf{R}_{\mathcal{O}_X(-H)} \colon D^b(X) \to D^b(X)$, there is the exact triangle
\[ F_x \to \mathbf{R}\mathrm{Hom}(F_x,\mathcal{O}_X(-H))^{\vee} \otimes \mathcal{O}_X(-H) \to \mathrm{pr}(\mathcal{O}_x). \]
We calculate $\mathbf{R}\mathrm{Hom}(F_x,\mathcal{O}_X(-H))$ in the next lemma.

\begin{lem}\label{thrd}
 We have $\mathbf{R}\mathrm{Hom}(F_x,\mathcal{O}_X(-H))=\mathbb{C}[-2]$.
\end{lem}
\begin{proof}
Applying $\mathbf{R}\mathrm{Hom}(-,\mathcal{O}_X(-H))$ to the exact sequences (\ref{ideal}) and (\ref{mid}), we have the exact triangles 
\[ \mathbf{R}\mathrm{Hom}(\mathcal{O}_x,\mathcal{O}_X(-H)) \to \mathbf{R}\mathrm{Hom}(\mathcal{O}_X(H),\mathcal{O}_X(-H)) \to \mathbf{R}\mathrm{Hom}(I_x(H),\mathcal{O}_X(-H)),\]
\[ \mathbf{R}\mathrm{Hom}(I_x(H),\mathcal{O}_X(-H)) \to \mathbf{R}\mathrm{Hom}(\mathcal{O}_X^{\oplus5},\mathcal{O}_X(-H)) \to \mathbf{R}\mathrm{Hom}(F_x,\mathcal{O}_X(-H)). \]
By Lemma \ref{sky} and $\mathbf{R}\mathrm{Hom}(\mathcal{O}_X(H),\mathcal{O}_X(-H))=0$,
 the first exact triangle is nothing but 
\[ \mathbb{C}[-4] \to 0 \to \mathbf{R}\mathrm{Hom}(I_x(H),\mathcal{O}_X(-H)).\]
So we obtain $\mathbf{R}\mathrm{Hom}(I_x(H),\mathcal{O}_X(-H))=\mathbb{C}[-3]$.

Since $\mathbf{R}\mathrm{Hom}(\mathcal{O}_X^{\oplus5},\mathcal{O}_X(-H))=0$, the second exact triangle is nothing but 
\[ \mathbf{R}\mathrm{Hom}(I_x(H),\mathcal{O}_X(-H)) \to 0 \to \mathbf{R}\mathrm{Hom}(F_x,\mathcal{O}_X(-H)). \]
This implies 

\begin{align*}
 \mathbf{R}\mathrm{Hom}(F_x,\mathcal{O}_X(-H)) &= \mathbf{R}\mathrm{Hom}(I_x(H),\mathcal{O}_X(-H))[1] \\
                                                                &= \mathbb{C}[-2].
\end{align*}
\end{proof}

By Lemma \ref{thrd}, we have the following exact triangle
\begin{equation}\label{proj}
 F_x \to \mathcal{O}_X(-H)[2] \to \mathrm{pr}(\mathcal{O}_x). 
\end{equation}
Collecting exact triangles (\ref{ideal}), (\ref{mid}) and (\ref{proj}), we have the following proposition.

\begin{lem}\label{res}
There are the following exact triangles on $X$:
\begin{equation}\label{fir}
I_x(H) \hookrightarrow \mathcal{O}_X(H) \twoheadrightarrow \mathcal{O}_x \stackrel{e_1}{\to} I_x(H)[1]
\end{equation}
\begin{equation}\label{sec}
F_x \hookrightarrow \mathcal{O}_X^{\oplus5} \twoheadrightarrow I_x(H) \stackrel{e_2}{\to} F_x[1]
\end{equation}
\begin{equation}\label{three}
F_x \stackrel{c}{\to} \mathcal{O}_X(-H)[2] \to \mathrm{pr}(\mathcal{O}_x) \stackrel{e_3}{\to} F_x.
\end{equation} 
Here $c$ is the morphism in \rm{(\ref{proj})}.
\end{lem}

Taking the long exact sequence of the exact triangle (\ref{three}), we obtain the following remark.
\begin{rem}\label{cohproj}
The following holds.
\begin{itemize}
\item $\mathcal{H}^0(\mathrm{pr}(\mathcal{O}_x))=F_x.$
\item $\mathcal{H}^{-1}(\mathrm{pr}(\mathcal{O}_x))=\mathcal{O}_X(-H).$
\item  $\mathcal{H}^k(\mathrm{pr}(\mathcal{O}_x))=0$ for any $k \neq -1,0$.
\end{itemize}
\end{rem}

The following proposition is the first statement in Proposition \ref{Lagproj}.
\begin{prop}

Let $x\neq y \in X$ be distinct points in $X$. Then $\mathrm{pr}(\mathcal{O}_x)$ is not isomorphic to $\mathrm{pr}(\mathcal{O}_y)$.
\end{prop}
\begin{proof}
By Remark \ref{cohproj}, it is sufficient to prove that $F_x$ is not isomorphic to $F_y$.
So we prove that $\mathcal{E}xt^2(F_x,\mathcal{O}_X) \simeq \mathcal{O}_x$.

Applying $\mathbf{R}\mathcal{H}om(-,\mathcal{O}_X)$ to the exact triangles (\ref{fir}) and (\ref{sec}), we can obtain the isomorphisms
\begin{align*}
 \mathcal{E}xt^2(F_x,\mathcal{O}_X) &\simeq \mathcal{E}xt^3(I_x(H),\mathcal{O}_X)\\
                                               &\simeq \mathcal{E}xt^4(\mathcal{O}_x,\mathcal{O}_X)\\
                                               &\simeq \mathcal{O}_x. 
\end{align*}
\end{proof}

Thus we have calculated the image $\mathrm{pr}(\mathcal{O}_x)$ of the skyscraper sheaf $\mathcal{O}_x$.

Second, we calculate $\mathrm{Ext}$-groups and prove the remaining statements in Proposition \ref{Lagproj}.

\begin{lem}\label{FIcoh}
The following holds.
\begin{itemize}
\item $\mathbf{R}\Gamma(X,F_x(H))=\mathbb{C}^{10}$.
\item $\mathbf{R}\mathrm{Hom}(I_x(H),\mathcal{O}_X)=\mathbb{C}[-3]$.
\end{itemize}
\end{lem}
\begin{proof}
  Consider the exact sequence
\[ 0 \to I_x(2H) \to \mathcal{O}_X(2H) \to \mathcal{O}_x \to 0 .\]
Taking $\mathbf{R}\Gamma(X,-)$, we have the exact triangle
\[ \mathbf{R}\Gamma(X,I_x(2H)) \to \mathbf{R}\Gamma(X,\mathcal{O}_X(2H)) \twoheadrightarrow \mathbf{R}\Gamma(X,\mathcal{O}_x). \]
Since $\mathbf{R}\Gamma(X, \mathcal{O}_X(2H))=\mathbb{C}^{21}$ and $\mathbf{R}\Gamma(X, \mathcal{O}_x)=\mathbb{C}$, we obtain \[\mathbf{R}\Gamma(X,I_x(2H))=\mathbb{C}^{20}.\]

Applying $\otimes \mathcal{O}_X(H)$ to the exact sequence (\ref{sec}), we have
\[ 0 \to F_x(H) \to \mathcal{O}_X(H)^{\oplus5} \to I_x(2H) \to 0. \]
Taking $\mathbf{R}\Gamma(X,-)$, we have the exact triangle
\[ \mathbf{R}\Gamma(X,F_x(H)) \to \mathbf{R}\Gamma(X,\mathcal{O}_X(H)^{\oplus5}) \twoheadrightarrow \mathbf{R}\Gamma(X,I_x(2H)). \]
Since $\mathbf{R}\Gamma(X,\mathcal{O}_X(H)^{\oplus5})=\mathbb{C}^{30}$ and  $\mathbf{R}\Gamma(X,I_x(2H))=\mathbb{C}^{20}$, we obtain 
\[\mathbf{R}\Gamma(X,F_x(H))=\mathbb{C}^{10}.\]

We prove the second claim.   Applying $\mathbf{R}\mathrm{Hom}(-,\mathcal{O}_X)$ to the exact sequence (\ref{fir}), we have the exact triangle
\[ \mathbf{R}\mathrm{Hom}(\mathcal{O}_x,\mathcal{O}_X) \to \mathbf{R}\mathrm{Hom}(\mathcal{O}_X(H),\mathcal{O}_X) \to \mathbf{R}\mathrm{Hom}(I_x(H),\mathcal{O}_X). \] 
By Lemma \ref{sky} and $\mathbf{R}\mathrm{Hom}(\mathcal{O}_X(H),\mathcal{O}_X)=0$, we obtain 
\[\mathbf{R}\mathrm{Hom}(I_x(H),\mathcal{O}_X)=\mathbb{C}[-3].\]
\end{proof}

\begin{lem}\label{iso}
There are the following isomorphisms.

\begin{equation}\label{u}
\circ e_1 \colon \mathbf{R}\mathrm{Hom}(I_x(H),I_x(H)) \stackrel{\sim}{\to}  \mathbf{R}\mathrm{Hom}(\mathcal{O}_x,I_x(H))[1]
\end{equation}

\begin{equation}\label{i}
\circ e_2 \colon \mathbf{R}\mathrm{Hom}(F_x,F_x) \stackrel{\sim}{\to} \mathbf{R}\mathrm{Hom}(I_x(H),F_x)[1]
\end{equation}

\begin{equation}\label{a}
 e_3 \circ \colon \mathbf{R}\mathrm{Hom}(\mathrm{pr}(\mathcal{O}_x),\mathrm{pr}(\mathcal{O}_x)) \stackrel{\sim}{\to}  \mathbf{R}\mathrm{Hom}(\mathrm{pr}(\mathcal{O}_x),F_x)[1] 
\end{equation}
\end{lem}
\begin{proof}
 Applying $\mathbf{R}\mathrm{Hom}(-,I_x(H))$ to the exact sequence (\ref{fir}), we have the exact triangle
\[ \mathbf{R}\mathrm{Hom}(\mathcal{O}_x,I_x(H)) \to \mathbf{R}\mathrm{Hom}(\mathcal{O}_X(H),I_x(H)) \to \mathbf{R}\mathrm{Hom}(I_x(H),I_x(H)).\]
Since $I_x(H) \in \langle \mathcal{O}_X(H) \rangle^{\bot}$, we have $\mathbf{R}\mathrm{Hom}(\mathcal{O}_X(H),I_x(H))=0$.
So we obtain the isomorphism
\[ \circ e_1 \colon \mathbf{R}\mathrm{Hom}(I_x(H),I_x(H)) \stackrel{\sim}{\to}  \mathbf{R}\mathrm{Hom}(\mathcal{O}_x,I_x(H))[1].\]

 Using $F_x \in \langle \mathcal{O}_X \rangle^{\bot}$ and $\mathrm{pr}(\mathcal{O}_x) \in \mathcal{A}_X$ similarly, we can obtain 
\[\circ e_2 \colon \mathbf{R}\mathrm{Hom}(F_x,F_x) \stackrel{\sim}{\to} \mathbf{R}\mathrm{Hom}(I_x(H),F_x)[1]\]
\[e_3 \circ \colon \mathbf{R}\mathrm{Hom}(\mathrm{pr}(\mathcal{O}_x),\mathrm{pr}(\mathcal{O}_x)) \stackrel{\sim}{\to}  \mathbf{R}\mathrm{Hom}(\mathrm{pr}(\mathcal{O}_x),F_x)[1].\] 
\end{proof}

Applying $\mathbf{R}\mathrm{Hom}(\mathcal{O}_x,-)$ to the exact triangle (\ref{fir}), we have the exact triangle
\begin{equation}\label{ka}
\mathbf{R}\mathrm{Hom}(\mathcal{O}_x,I_x(H)) \to  \mathbf{R}\mathrm{Hom}(\mathcal{O}_x,\mathcal{O}_X(H)) \to \mathbf{R}\mathrm{Hom}(\mathcal{O}_x,\mathcal{O}_x)
\end{equation}
By Lemma \ref{sky}, the exact triangle (\ref{ka}) is nothing but 
\begin{equation}\label{kaa}
 \mathbf{R}\mathrm{Hom}(\mathcal{O}_x,I_x(H)) \to \mathbb{C}[-4] \to \mathbf{R}\mathrm{Hom}(\mathcal{O}_x,\mathcal{O}_x).
\end{equation}

Taking the long exact sequence of the exact triangle (\ref{kaa}), we have the folloing isomorphisms.

\begin{lem}\label{e1}
There are the following isomorphisms.
\[e_1 \circ \colon \mathrm{Hom}(\mathcal{O}_x,\mathcal{O}_x) \stackrel{\sim}{\to} \mathrm{Ext}^1(\mathcal{O}_x,I_x(H))\]
\[e_1 \circ \colon \mathrm{Ext}^1(\mathcal{O}_x,\mathcal{O}_x) \stackrel{\sim}{\to} \mathrm{Ext}^2(\mathcal{O}_x,I_x(H))\]
\[e_1 \circ \colon \mathrm{Ext}^2(\mathcal{O}_x,\mathcal{O}_x) \stackrel{\sim}{\to} \mathrm{Ext}^3(\mathcal{O}_x,I_x(H))\]
\end{lem}

Applying $\mathbf{R}\mathrm{Hom}(I_x(H),-)$ to the exact triangle (\ref{sec}), we have the exact triangle
\begin{equation}\label{o}
\mathbf{R}\mathrm{Hom}(I_x(H),F_x) \to \mathbf{R}\mathrm{Hom}(I_x(H),\mathcal{O}_X^{\oplus5}) \to \mathbf{R}\mathrm{Hom}(I_x(H),I_x(H))
\end{equation}
By Lemma \ref{FIcoh}, the exact triangle (\ref{o}) is nothing but 
\begin{equation}\label{oa}
\mathbf{R}\mathrm{Hom}(I_x(H),F_x) \to \mathbb{C}^5[-3] \to \mathbf{R}\mathrm{Hom}(I_x(H),I_x(H)).
\end{equation}

Taking the long exact sequence of the exact triangle (\ref{oa}), we have the following isomorphisms.

\begin{lem}\label{e2}
There are the following isomorphisms.
\[e_2 \circ \colon \mathrm{Hom}(I_x(H),I_x(H)) \stackrel{\sim}{\to} \mathrm{Ext}^1(I_x(H), F_x)\]
\[e_2 \circ \colon \mathrm{Ext}^1(I_x(H),I_x(H)) \stackrel{\sim}{\to} \mathrm{Ext}^2(I_x(H), F_x)\]
\end{lem}

Applying $\mathbf{R}\mathrm{Hom}(-,F_x)$ to the exact triangle (\ref{three}), we have the exact triangle
\begin{equation}\label{e}
\mathbf{R}\mathrm{Hom}(\mathrm{pr}(\mathcal{O}_x),F_x) \to \mathbf{R}\mathrm{Hom}(\mathcal{O}_X(-H)[-2],F_x) \to \mathbf{R}\mathrm{Hom}(F_x,F_x).
\end{equation}
By Lemma \ref{FIcoh}, the exact triangle (\ref{e}) is nothing but 
\begin{equation}\label{ea}
\mathbf{R}\mathrm{Hom}(\mathrm{pr}(\mathcal{O}_x),F_x) \to \mathbb{C}^{10}[-2] \to \mathbf{R}\mathrm{Hom}(F_x,F_x).
\end{equation}

Taking the lomg exact sequence of the exact triangle (\ref{ea}), we have the following isomorphism.

\begin{lem}\label{e3}
There is the isomorphism
\[\circ e_3 \colon \mathrm{Hom}(F_x,F_x) \stackrel{\sim}{\to} \mathrm{Ext}^1(\mathrm{pr}(\mathcal{O}_x),F_x).\]
\end{lem}

We can prove that the object $\mathrm{pr}(\mathcal{O}_x)$ is simple.
\begin{cor}\label{simple}
We have $\mathrm{Hom}(\mathrm{pr}(\mathcal{O}_x),\mathrm{pr}(\mathcal{O}_x))=\mathbb{C}$.
\end{cor}
\begin{proof}
By Lemma \ref{iso}, Lemma \ref{e1}, Lemma \ref{e2} and Lemma \ref{e3}, we have the isomorphisms
\begin{align*}
 \mathrm{Hom}(\mathrm{pr}(\mathcal{O}_x),\mathrm{pr}(\mathcal{O}_x)) &\stackrel{e_3 \circ}{\simeq} \mathrm{Ext}^1(\mathrm{pr}(\mathcal{O}_x),F_x) \\
 &\stackrel{\circ e_3}{\simeq} \mathrm{Hom}(F_x,F_x)\\
 &\stackrel{\circ e_2}{\simeq}  \mathrm{Ext}^1(I_x(H),F_x)\\
 & \stackrel{e_2 \circ}{\simeq} \mathrm{Hom}(I_x(H),I_x(H))\\
 &\stackrel{\circ e_1}{\simeq} \mathrm{Ext}^1(\mathcal{O}_x,I_x(H))\\
 &\stackrel{e_1 \circ}{\simeq} \mathrm{Hom}(\mathcal{O}_x,\mathcal{O}_x)=\mathbb{C}. 
\end{align*}
\end{proof}

\begin{lem}
We have $\mathrm{Ext}^2(F_x,F_x)=\mathbb{C}^7$.
\end{lem}
\begin{proof}
By the exact triangle (\ref{oa}), we have the exact sequence
\begin{align*}
 0 \to \mathrm{Ext}^2(I_x(H),I_x(H)) &\to \mathrm{Ext}^3(I_x(H),F_x) \to \mathbb{C}^5\\ \to \mathrm{Ext}^3(I_x(H),I_x(H)) &\to \mathrm{Ext}^4(I_x(H),F_x) \to 0. 
\end{align*}
By Lemma \ref{iso} and the isomorphism (\ref{a}), we have 
\begin{align*}
\mathrm{Ext}^4(I_x(H),F_x) &\simeq \mathrm{Ext}^3(F_x,F_x)\\
&\simeq \mathrm{Ext}^4(\mathrm{pr}(\mathcal{O}_x),F_x)\\
                               &\simeq \mathrm{Ext}^3(\mathrm{pr}(\mathcal{O}_x),\mathrm{pr}(\mathcal{O}_x))=0.
\end{align*}
By Lemma \ref{iso} and Lemma \ref{e1}, we have
\[\mathrm{Ext}^2(I_x(H),I_x(H))=\mathbb{C}^6\]
\[\mathrm{Ext}^3(I_x(H),I_x(H))=\mathbb{C}^4\]
\[\mathrm{Ext}^2(F_x,F_x) \simeq \mathrm{Ext}^3(I_x(H),F_x).\]
So the above long exact sequence can be described as
\[ 0 \to \mathbb{C}^6 \to \mathrm{Ext}^2(F_x,F_x) \to \mathbb{C}^5\\ \to \mathbb{C}^4 \to 0. \]
Hence, we have $\mathrm{Ext}^2(F_x,F_x)=\mathbb{C}^7$.
\end{proof}

We can calculate the dimension of $\mathrm{Ext}^1(\mathrm{pr}(\mathcal{O}_x),\mathrm{pr}(\mathcal{O}_x))$.

\begin{cor}\label{tangent}
We have $\mathrm{Ext}^1(\mathrm{pr}(\mathcal{O}_x),\mathrm{pr}(\mathcal{O}_x))=\mathbb{C}^8$.
\end{cor}
\begin{proof}
By Lemma \ref{iso} and Lemma \ref{e2}, we have 
\begin{equation}\label{1FF}
\mathrm{Ext}^1(F_x,F_x)=\mathbb{C}^4
\end{equation}
\[\mathrm{Ext}^1(\mathrm{pr}(\mathcal{O}_x), \mathrm{pr}(\mathcal{O}_x)) \simeq \mathrm{Ext}^2(\mathrm{pr}(\mathcal{O}_x),F_x).\] 
Moreover, using Lemma \ref{iso} and Corollary \ref{simple}, we have 
\[\mathrm{Ext}^3(\mathrm{pr}(\mathcal{O}_x),F_x) \simeq \mathrm{Ext}^2(\mathrm{pr}(\mathcal{O}_x), \mathrm{pr}(\mathcal{O}_x))=\mathbb{C}.\]
Here the last equality is deduced from the Serre duality for $\mathcal{A}_X$.
By the exact triangle (\ref{ea}), we obtain the long exact sequence
\[0 \to \mathbb{C}^4 \to \mathrm{Ext}^1(\mathrm{pr}(\mathcal{O}_x), \mathrm{pr}(\mathcal{O}_x)) \to \mathbb{C}^{10} \to \mathbb{C}^7 \to \mathbb{C} \to 0.\]
So we obtain $\mathrm{Ext}^1(\mathrm{pr}(\mathcal{O}_x), \mathrm{pr}(\mathcal{O}_x))=\mathbb{C}^8$.
\end{proof}

We will complete the proof of the third statement in Proposition \ref{Lagproj}.

\begin{prop}\label{closedimm}
 The linear map
\[ \mathrm{pr} \colon \mathrm{Ext}^1(\mathcal{O}_x,\mathcal{O}_x) \to \mathrm{Ext}^1(\mathrm{pr}(\mathcal{O}_x),\mathrm{pr}(\mathcal{O}_x))\]
is injective.
\end{prop}
\begin{proof}
By Lemma \ref{iso}, Lemma \ref{e1}, Lemma \ref{e2} and Lemma \ref{e3}, the linear map
\[ \mathrm{pr} \colon \mathrm{Ext}^1(\mathcal{O}_x,\mathcal{O}_x) \to \mathrm{Ext}^1(\mathrm{pr}(\mathcal{O}_x),\mathrm{pr}(\mathcal{O}_x))\]
can be factorized as follows:

\begin{align*}
\mathrm{pr} \colon \mathrm{Ext}^1(\mathcal{O}_x,\mathcal{O}_x) &\stackrel{e_1 \circ}{\simeq}  \mathrm{Ext}^2(\mathcal{O}_x,I_x(H))\\ &\stackrel{\circ e_1}{\simeq} \mathrm{Ext}^1(I_x(H),I_x(H))\\
  &\stackrel{e_2 \circ}{\simeq} \mathrm{Ext}^2(I_x(H),F_x)\\
  &\stackrel{\circ e_2}{\simeq} \mathrm{Ext}^1(F_x,F_x)\\
  &\stackrel{\circ e_3}{\hookrightarrow} \mathrm{Ext}^2(\mathrm{pr}(\mathcal{O}_x),F_x)\\
   &\stackrel{e_3 \circ}{\simeq} \mathrm{Ext}^1(\mathrm{pr}(\mathcal{O}_x),\mathrm{pr}(\mathcal{O}_x)).
\end{align*}

\end{proof}

Finaly, we will prove the last statement in Proposition \ref{Lagproj}.
Before giving a proof, we recall the definition of the bilinear form on $\mathrm{Ext}^1(\mathrm{pr}(\mathcal{O}_x),\mathrm{pr}(\mathcal{O}_x))$, which is corresponding to the symplectic forms on moduli spaces of Bridgeland stable complexes on twisted K3 surfaces.

\begin{dfn}
We define a bilinear form 
\[\omega_x \colon \mathrm{Ext}^1(\mathrm{pr}(\mathcal{O}_x),\mathrm{pr}(\mathcal{O}_x)) \times \mathrm{Ext}^1(\mathrm{pr}(\mathcal{O}_x),\mathrm{pr}(\mathcal{O}_x)) \to \mathbb{C}\] by the composition of morphisms in the derived category.
\end{dfn}

The following proposition implies Proposition \ref{main}(d).
\begin{prop}
The bilinear form $\omega_x$ vanishes on $\mathrm{Ext}^1(\mathcal{O}_x,\mathcal{O}_x) \times \mathrm{Ext}^1(\mathcal{O}_x,\mathcal{O}_x)$.
\end{prop}
\begin{proof}
Consider the following commutative diagram:

\begin{xy}
\xymatrix{\mathrm{Ext}^1(\mathcal{O}_x,\mathcal{O}_x) \times \mathrm{Ext}^1(\mathcal{O}_x,\mathcal{O}_x) \ar[r]^-{\mathrm{composition}} \ar[d]_{\mathrm{pr}} &   \mathrm{Ext}^2(\mathcal{O}_x,\mathcal{O}_x) \ar[d]_{\mathrm{pr}}\\
\mathrm{Ext}^1(\mathrm{pr}(\mathcal{O}_x),\mathrm{pr}(\mathcal{O}_x)) \times \mathrm{Ext}^1(\mathrm{pr}(\mathcal{O}_x),\mathrm{pr}(\mathcal{O}_x))   \ar[r]_-{\omega_x} &\mathrm{Ext}^2(\mathrm{pr}(\mathcal{O}_x),\mathrm{pr}(\mathcal{O}_x)).}
\end{xy}

It is sufficient to prove that  
\begin{equation}\label{ext2}
\mathrm{pr} \colon \mathrm{Ext}^2(\mathcal{O}_x,\mathcal{O}_x) \to \mathrm{Ext}^2(\mathrm{pr}(\mathcal{O}_x),\mathrm{pr}(\mathcal{O}_x))
\end{equation} 
is  zero.

The linear map (\ref{ext2}) can be factorized as follows:
\begin{align*}
\mathrm{pr} \colon  \mathrm{Ext}^2(\mathcal{O}_x,\mathcal{O}_x) &\stackrel{e_1 \circ}{\simeq} \mathrm{Ext}^3(\mathcal{O}_x,I_x(H))\\
                                                                                    &\stackrel{\circ e_1}{\simeq} \mathrm{Ext}^2(I_x(H),I_x(H))\\
                                                                                    &\stackrel{e_2 \circ}{\hookrightarrow} \mathrm{Ext}^3(I_x(H),F_x)\\
                                                                                    &\stackrel{\circ e_2}{\simeq} \mathrm{Ext}^2(F_x,F_x)\\
                                                          &\stackrel{\circ e_3}{\twoheadrightarrow} \mathrm{Ext}^3(\mathrm{pr}(\mathcal{O}_x),F_x)\\
                                              &\stackrel{e_3 \circ}{\simeq} \mathrm{Ext}^2(\mathrm{pr}(\mathcal{O}_x),\mathrm{pr}(\mathcal{O}_x)).
\end{align*}

Applying $\mathbf{R}\mathrm{Hom}(-,\mathcal{O}_X)$ to the exact triangle (\ref{sec}), we have the exact triangle
\[  \mathbf{R}\mathrm{Hom}(I_x(H),\mathcal{O}_X) \to \mathbf{R}\mathrm{Hom}(\mathcal{O}_X^{\oplus5},\mathcal{O}_X) \to \mathbf{R}\mathrm{Hom(F_x,\mathcal{O}_X)}. \]
Since $\mathbf{R}\mathrm{Hom}(\mathcal{O}_X^{\oplus5},\mathcal{O}_X)=\mathbb{C}^5$, we have
\begin{equation}\label{FIext}
 \mathrm{Ext}^2(F_x, \mathcal{O}_X^{\oplus5}) \stackrel{\circ e_2}{\simeq} \mathrm{Ext}^3(I_x(H),\mathcal{O}_X^{\oplus5}). 
\end{equation}

By the isomorphism (\ref{FIext}) and the exact triangle (\ref{o}), we have 
\begin{align*}
&\mathrm{Im}(\mathrm{Ext}^2(I_x(H),I_x(H)) \hookrightarrow \mathrm{Ext}^2(F_xF_x))\\
=&\mathrm{Ker}(\mathrm{Ext}^2(F_x,F_x) \to \mathrm{Ext}^2(F_x,\mathcal{O}_X^{\oplus5})). 
\end{align*}
Note that this vector space is $6$-dimensional.

Recall that $c \colon F_x \to \mathcal{O}_X(-H)[2]$ is the morphism in the exact triangle (\ref{three}). Taking the long exact seqence of the exact triangle (\ref{e}), we have the following exact sequence
\begin{align*}
0 &\to \mathrm{Ext}^1(F_x,F_x) \stackrel{\circ e_3}{\to}  \mathrm{Ext}^2(\mathrm{pr}(\mathcal{O}_x),F_x) \to \mathrm{Ext}^2(\mathcal{O}_X(-H)[2],F_x) \\
&\stackrel{\circ c}{\to}  \mathrm{Ext}^2(F_x,F_x)  \stackrel{\circ e_3}{\twoheadrightarrow}  \mathrm{Ext}^3(\mathrm{pr}(\mathcal{O}_x),F_x) \to 0. 
\end{align*}
Hence, we have 
\begin{align*}
 &\mathrm{Ker}(\mathrm{Ext}^2(F_x,F_x)\stackrel{\circ e_3}{\twoheadrightarrow} \mathrm{Ext}^3(F_x,\mathrm{pr}(\mathcal{O}_x)))\\=&\mathrm{Im}(\mathrm{Ext}^2(\mathcal{O}_X(-H)[2],F_x) \stackrel{\circ c}{\to} \mathrm{Ext}^2(F_x,F_x)) . 
\end{align*}
By (\ref{1FF}), Lemma \ref{FIcoh}, Lemma \ref{iso} and Corollary \ref{tangent}, this vector space is $6$-dimensional.

So it is enough to prove that
\[ \mathrm{Im}(\mathrm{Ext}^2(\mathcal{O}_X(-H)[2],F_x) \stackrel{\circ c}{\to} \mathrm{Ext}^2(F_x,F_x)) \subset \mathrm{Ker}(\mathrm{Ext}^2(F_x,F_x) \to \mathrm{Ext}^2(F_x,\mathcal{O}_X^{\oplus5})). \]

Take $\psi \in \mathrm{Im}(\mathrm{Ext}^2(\mathcal{O}_X(-H)[2],F_x) \stackrel{\circ c}{\to} \mathrm{Ext}^2(F_x,F_x))$.

Then there is a morphism $\eta \in \mathrm{Ext}^2(\mathcal{O}_X(-H)[2], F_x)$ satisfying  a following commutative diagram:
\[ \xymatrix{{F_x} \ar[d]^{\psi} \ar[r]^{c} & \mathcal{O}_X(-H)[2] \ar[ld]^{\eta} \\ F_x[2] \ar[r] & \mathcal{O}_X^{\oplus5}[2] .}  \]

Take a hyperplane section $H$ of $X$ such that $x \notin H$. Let $i \colon \mathcal{O}_X(-H) \to \mathcal{O}_X$ be the morphism defining $H$.

We prove that $i[2] \circ c \neq 0$. Assume that $i[2] \circ c=0$. Then there is a morphism between exact triangles:
\[ \xymatrix{F_x \ar[r]^c \ar[d] &  \mathcal{O}_X(-H) \ar[r] \ar[d]^{\mathrm{id}} &
\mathrm{pr}(\mathcal{O}_x) \ar[d] \\
\mathcal{O}_H(1)[1] \ar[r] & \mathcal{O}_X(-H)[2] \ar[r]^{i[2]} & \mathcal{O}_X[2] . }\]

Since $x \notin H$, we have 
\[ 0 \to F_x|_H \to \mathcal{O}_H^{\oplus5} \to \mathcal{O}_H(1) \to 0,\]
which is the restriction of the exact sequence (\ref{fir}). 

Applying $\mathbf{R}\mathrm{Hom}(-,\mathcal{O}_H)$ to this exact seqence, we have
\[ \mathbf{R}\mathrm{Hom}(F_x,\mathcal{O}_H)=\mathbf{R}\mathrm{Hom}(F_x|H,\mathcal{O}_H)=\mathbb{C}^5. \]

This implies $\mathrm{Ext}^1(F_x,\mathcal{O}_H)=0$. Since $c \neq 0$, this is contradiction.

Note that the vector space $\mathrm{Ker}(\mathrm{Ext}^2(\mathcal{O}_X(-H)[2],\mathcal{O}_X) \stackrel{\circ c}{\twoheadrightarrow} \mathrm{Ext}^2(F_x,\mathcal{O}_X))$ is generated by morphisms $\mathcal{O}_X(-H) \to \mathcal{O}_X$ induced by hyperplane sections of $X$, which is through the point $x$. By the definition of $F_x$, the composition $\mathcal{O}_X(-H) \stackrel{\eta}{\to}F_x \to \mathcal{O}_X^{\oplus5}$ are induced by hyperplane sections of $X$, which is through $x \in X$.  So the composition $F_x \stackrel{\psi}{\to} F_x[2] \to \mathcal{O}_X^{\oplus5}$ is zero.
Hence, we have 
\[\psi \in  \mathrm{Ker}(\mathrm{Ext}^2(F_x,F_x) \to \mathrm{Ext}^2(F_x,\mathcal{O}_X^{\oplus5})). \]
\end{proof}

Thus we have proved Proposition \ref{Lagproj}.
In the next section, we will see properties of the object $P_x$ on the twisted K3 surface, which is corresponding to the point $x \in X$.

\section{DESCRIPTION OF COMPLEXES ON TWISTED K3 SURFACES}
Let $X$ be a cubic fourfold containing a plane $P$ as in Section 2.4 and $(S,\alpha)$ be the correspoding twisted K3 surface. We use the same notation  as Section 2.4.

\begin{dfn}
For a point $x \in X$, we define the object $R_x \in D^b(S,\alpha)$ as follow:
\[ R_x:=\Phi_{\mathbb{P}^2}^{-1}(\mathrm{pr}(\mathcal{O}_x))[-4] \in D^b(\mathbb{P}^2,\mathcal{B}_0). \]
\end{dfn}

\begin{lem}
Let $x \in X$ be a point. Then the followings hold.
\begin{itemize}
\item[\rm(a)]$R_x \simeq  \Psi(\mathbf{L}\sigma^{*}(I_x(H))[-2]$.
\item[\rm(b)] There is the exact triangle:
\begin{equation}\label{sankaku}
R_x \to \mathcal{B}_0(h) \to \Psi(\mathbf{L}\sigma^{*}\mathcal{O}_x)[-2]. 
\end{equation}
\end{itemize}
\end{lem}
\begin{proof}
\begin{itemize}
\item[(a)] Since $\mathbf{L}\sigma^{*} \colon D^b(\tilde{X}) \to D^b(X)$ is fully faithful, we have
\begin{align*}
R_x&=\Phi_{\mathbb{P}^2}^{-1}(\mathrm{pr}(\mathcal{O}_x))[-4]\\
     &=\Psi\mathbf{L}_{\mathcal{O}_{\tilde{X}}(-h)}\mathbf{R}_{\mathcal{O}_{\tilde{X}}(h-H)}\mathbf{L}\sigma^{*}\mathbf{R}_{\mathcal{O}_X(-H)}\mathbf{L}_{\mathcal{O}_X}\mathbf{L}_{\mathcal{O}_X(H)}(\mathcal{O}_x)[-3]\\
&\simeq \Psi\mathbf{L}_{\mathcal{O}_{\tilde{X}}(-h)}\mathbf{R}_{\mathcal{O}_{\tilde{X}}(h-H)}\mathbf{R}_{\mathcal{O}_{\tilde{X}}(-H)}\mathbf{L}_{\mathcal{O}_{\tilde{X}}}\mathbf{L}\sigma^{*}(I_x(H))[-2]. 
\end{align*}
First, we prove that $\Psi\mathbf{L}_{\mathcal{O}_{\tilde{X}}(-h)}(E) \simeq \Psi(E)$ for any $E \in D^b(\tilde{X})$. By the definition of mutation functors, there is the exact triangle:
\[ \mathbf{R}\mathrm{Hom}(\mathcal{O}_{\tilde{X}}(-h),E)\otimes\mathcal{O}_{\tilde{X}}(-h) \to E \to \mathbf{L}_{\mathcal{O}_{\tilde{X}}(-h)}(E). \]
Applying the functor $\Psi$, we have the exact triangle
\[ \Psi(\mathbf{R}\mathrm{Hom}(\mathcal{O}_{\tilde{X}}(-h),E)\otimes\mathcal{O}_{\tilde{X}}(-h)) \to \Psi(E) \to \Psi(\mathbf{L}_{\mathcal{O}_{\tilde{X}}(-h)}(E)). \]
By Lemma \ref{reducelem}, we have
\[\Psi(\mathbf{R}\mathrm{Hom}(\mathcal{O}_{\tilde{X}}(-h),E)\otimes\mathcal{O}_{\tilde{X}}(-h))= \mathbf{R}\mathrm{Hom}(\mathcal{O}_{\tilde{X}}(-h),E)\otimes\Psi(\mathcal{O}_{\tilde{X}}(-h))=0.\]
So we have $\Psi\mathbf{L}_{\mathcal{O}_{\tilde{X}}(-h)}(E) \simeq \Psi(E)$.

Imitating these arguments, we have the isomorphism $R_x \simeq \Psi(\mathbf{L}\sigma^{*}(I_x(H))[-2]$.

\item[(b)] Applying $\Psi(\mathbf{L}\sigma^{*}(-))[-2]$ to the exact triangle (\ref{fir}),
we have the exact triangle:
\[R_x \to \Psi(\mathbf{L}\sigma\mathcal{O}_X(H))[-2] \to \Psi(\mathbf{L}\sigma^{*}\mathcal{O}_x)[-2]. \]
By Lemma \ref{reducelem}, we have the isomorphisms
\begin{align*}
\Psi(\mathbf{L}\sigma\mathcal{O}_X(H))[-2] &\simeq  \Psi(\mathcal{O}_{\tilde{X}}(H))[-2]\\
                                                         &\simeq \Psi(\mathcal{O}_{\tilde{X}}(-h+H))(h)[-2]\\
                                                         &\simeq \mathcal{B}_0(h).
\end{align*}
Hence, we have obtained the desired exact triangle.
\end{itemize}
\end{proof}

If $x \in P$, then we have $\mathcal{H}^{-1}(\mathbf{L}\sigma^*\mathcal{O}_x)=\mathcal{O}_{\sigma^{-1}(x)}(D)$,  $\mathcal{H}^0(\mathbf{L}\sigma^{*}\mathcal{O}_x)=\mathcal{O}_{\sigma^{-1}(x)}$ and the others are zero. Since $D=H-h$, we obtain the following lemma.

\begin{lem}\label{coh}
The followings hold.
\begin{itemize}
\item If $x \in X \setminus P$, then $\Psi(\mathbf{L}\sigma^{*}\mathcal{O}_x)[-2]= \pi_*(\mathcal{E}(h))|_{\pi(\sigma^{-1}(x))}$. 
\item If $x \in P$, then we have
 \[\mathcal{H}^0(\Psi(\mathbf{L}\sigma^{*}\mathcal{O}_x)[-2])=\pi_*(\mathcal{E}(h))|_{\pi(\sigma^{-1}(x))}\] 
\[ \mathcal{H}^{-1}(\Psi(\mathbf{L}\sigma^{*}\mathcal{O}_x)[-2])=\pi_*(\mathcal{E})|_{\pi(\sigma^{-1}(x))}\]
 and others are zero.
\end{itemize}
\end{lem}
\begin{lem}
Let $x \in X$ be a point. Then followings hold.
\begin{itemize}
\item[\rm{(a)}]The object $R_x$ is a sheaf.
\item[\rm{(b)}]Assume that $x \in X \setminus P$. Taking the long exact sequence of the exact triangle  $(\mathrm{\ref{sankaku}})$, we have the exact sequence
\[ 0 \to R_x \to \mathcal{B}_0(h) \to \pi_*(\mathcal{E}(h))|_{\pi(\sigma^{-1}(x))} \to 0 .\]
Here $\pi(\sigma^{-1}(x))$ is a point in $\mathbb{P}^2$.
\item[\rm{(c)}]Assume that $x \in P$. Taking the long exact sequence of the exaxt triangle $(\mathrm{\ref{sankaku}})$, we have the exact sequence
\[ 0 \to (R_x)_{\mathrm{tor}} \to R_x \to \mathcal{B}_0(h) \to \pi_{*}(\mathcal{E}(h))|_{\pi(\sigma^{-1}(x))} \to 0. \]
Here $\pi(\sigma^{-1}(x))$ is a line in $\mathbb{P}^2$ and $R_x$ is an 1-dimensional pure torsion sheaf.
\end{itemize}
\end{lem}
\begin{proof}
Take a point $x \in X$.
By Lemma \ref{coh}, it is sufficient to prove that the morphism $\mathcal{B}_0(h) \to \pi_*(\mathcal{E}(h))|_{\pi(\sigma^{-1}(x))}$ is surjective.

Restricting the exact sequence (\ref{surj}) to $\tilde{X}$, we have the surjection $\pi^{*}\mathcal{B}_0(-H) \twoheadrightarrow \mathcal{E}$. So we can obtain the surjective morphism  $\pi^{*}\mathcal{B}_0(h-H) \twoheadrightarrow \mathcal{E}(h)$.

Note that $\pi|_{\sigma^{-1}(x)} \colon \sigma^{-1}(x) \to \mathbb{P}^2$ is a closed immersion. Restricting the morphism $\pi^{*}\mathcal{B}_0(h-H) \twoheadrightarrow \mathcal{E}(h)$ to $\sigma^{-1}(x)$ and the taking direct images of $\pi$, we have the surjective morphism $\mathcal{B}_0(h-H)|_{\pi(\sigma^{-1}(x))} \twoheadrightarrow \mathcal{E}(h)|_{\pi(\sigma^{-1}(x))}$. Now we can ignore $\otimes \mathcal{O}_X(-H)$. 
So there is the following commutative diagram.

\[\xymatrix{{\mathcal{B}_0(h)} \ar@{->>}[d]^{\rm{restriction}} \ar[rd] & \\ {\mathcal{B}_0(h)|_{\pi(\sigma^{-1}(x))}} \ar@{->>}[r] & {\pi_{*}(\mathcal{E}(h))|_{\pi(\sigma^{-1}(x))}}} \]
So the morphism $\mathcal{B}_0(h) \to \pi_*(\mathcal{E}(h))|_{\pi(\sigma^{-1}(x))}$ is surjective.

\end{proof}

Considering these exact sequences on the twisted K3 surface $(S,\alpha)$, we have the following proposition.
\begin{prop}
Let $x \in X$ be a point. Then there is the exact triangle:
\begin{equation}\label{basictri}
 P_x \to \mathcal{U}_0 \to Q_x. 
\end{equation}
Here $Q_x:=(f_*(- \otimes \mathcal{U}_0^{\vee}))^{-1}(\Psi(\mathbf{L}\sigma^{*}\mathcal{O}_x)[-2]) \in D^b(S,\alpha)$.

If $x \in X \setminus P$, then $Q_x$ is a zero dimensional torsion sheaf of length 2 and the exact triangle {\rm{(\ref{basictri})}} induces a following exact sequence in $\mathrm{Coh}(S,\alpha)$
\begin{equation}\label{Mukai}
0 \to P_x \to \mathcal{U}_0 \to Q_x \to 0.
\end{equation}

If $x \in P$, then the exact triangle {\rm{(\ref{basictri})}} induces the following exact sequence in $\mathrm{Coh}(S,\alpha)$
\begin{equation}
0 \to (P_x)_{\mathrm{tor}} \to P_x \to \mathcal{U}_0 \to \mathcal{H}^0(Q_x) \to 0.
\end{equation}
Here $(P_x)_{\mathrm{tor}}=\mathcal{H}^0(Q_x)(-h)$ is an 1-dimensional pure torsion sheaf.
\end{prop}
Next, we calculate Mukai vectors.
Fix a $B$-field $B \in \mathrm{H}^2(S,\frac{1}{2}\mathbb{Z})$ of the Brauer class $\alpha$.

\begin{lem}[\cite{Tod13}, Lemma 4.6]

We can describe $v^B(\mathcal{U}_0)=(2,s,t)$ such that $s^2-4t=-2$ and $s-2B \in \mathrm{Pic}S$.
\end{lem}
\begin{proof}
Recall that $\mathcal{U}_0$ is the $\alpha$-twisted vector bundle of rank $2$.  So we can write $v^B(\mathcal{U}_0)=(2,s,t) \in \tilde{\mathrm{H}}^{1,1}(S,B,\mathbb{Z})$.  Since $\mathcal{U}_0$ is spherical,   we have
 $\chi(\mathcal{U}_0,\mathcal{U}_0)=-2$.  By the Riemann-Roch formula (\ref{RR}), we have $s^2-4t=-2.$

By Lemma \ref{gene}, we have $s-2B \in \mathrm{Pic}S$.
\end{proof}
Toda \cite{Tod13}(Corollary 4.4) proved that
\[[\mathcal{B}_1]=\frac{3}{8}[\mathcal{B}_0]+\frac{3}{4}[\mathcal{B}_0(h)]-\frac{1}{8}[\mathcal{B}_0(2h)]\] 
in $N(D^b(\mathbb{P}^2,\mathcal{B}_0))$.
Let $\mathcal{U}_1 \in \mathrm{Coh}(S,\alpha)$ be the $\alpha$-twisted vector bundle corresponding to $\mathcal{B}_1$.
Using this relation, we can calculate the Mukai vector of $\mathcal{U}_1$ as follow.

\begin{lem}[\cite{Tod13}, Lemma 4.6]
We have 
\[v^B(\mathcal{U}_1)=e^{h/2}v^B(\mathcal{U}_0)=\left(s,s+h,t+\frac{1}{2}sh+\frac{1}{2}\right).\]
\end{lem}

We calculate the mukai vector of $P_x$.
\begin{prop}
Let $x \in X$ be a point. Then 
\begin{equation}\label{mukai}
v^B(P_x)=(2,s+2h,t+sh).
\end{equation}
\end{prop}

\begin{proof}
By Lemma \ref{res}, the numerical classes of $P_x$ and $P_y$ are same for any points $x,y \in X$. So we can assume that $x \in X \setminus P$. Since $Q_x$ is a zero dimensional torsion sheaf of length 2, we have $v^B(Q_x)=(0,0,2)$.
Using the exact sequence (\ref{Mukai}), we have 
\begin{align*}
v^B(P_x)&=v^B(\mathcal{U}_0(h))-v^B(Q_x)\\
           &=e^h(2,s,t)-(0,0,2)\\
           &=(2,s+2h,t+sh).
\end{align*}
\end{proof}

In the next lemma, we calculate the Mukai vector of $(P_x)_{\mathrm{tor}}$.

\begin{lem}
Let $x \in P$ be a point. Then we have
\begin{equation}\label{mukaitor}
v^B((P_x)_{\mathrm{tor}})=\left(0,h,\frac{1}{2}sh-\frac{1}{2}\right).
\end{equation}
\end{lem}
\begin{proof}
Take a line $C_x \subset X$, which is through a point $x$.

Let $C_x^{\prime}:=\sigma^{-1}(C_x)$ and $l_x:=\sigma^{-1}(x)$.
Then there are isomorphisms
\[ \mathcal{O}_D(-C_x^{\prime}) \simeq \mathcal{O}_D(-H)\]
\[ \mathcal{O}_{C_x^{\prime}}(-l_x) \simeq \mathcal{O}_{C_x^{\prime}}(-H).\]

Consider the following exact sequences:
\begin{equation}\label{koko1}
0 \to \mathcal{O}_{\tilde{X}}(-D)(=\mathcal{O}_{\tilde{X}}(h-H)) \to \mathcal{O}_{\tilde{X}} \to \mathcal{O}_D \to 0
\end{equation}
\begin{equation}\label{koko2}
0 \to \mathcal{O}_D(-C_x^{\prime})(=\mathcal{O}_D(-H)) \to \mathcal{O}_D \to \mathcal{O}_{C_x^{\prime}} \to 0
\end{equation}
\begin{equation}\label{koko}
0 \to \mathcal{O}_{C_x^{\prime}} \to \mathcal{O}_{C_x^{\prime}}(H) \to \mathcal{O}_{l_x} \to 0.
\end{equation}
Here the exact sequence (\ref{koko}) is induced by the exact sequence
\[ 0 \to \mathcal{O}_{C_x^{\prime}}(-l_x)(=\mathcal{O}_{C_x^{\prime}}(-H)) \to \mathcal{O}_{C_x^{\prime}} \to \mathcal{O}_{l_x} \to 0.\]
(Note that $\mathcal{O}_{l_x}(H) \simeq \mathcal{O}_{l_x}$.)

Applying the functor $\Psi(-)[-2]$ to (\ref{koko}), we have the exact triangle:
\[ \Psi(\mathcal{O}_{C_x^{\prime}})[-2] \to \Psi(\mathcal{O}_{C_x^{\prime}}(H))[-2] \to \Psi(\mathcal{O}_{l_x})[-2].\]

If $\Psi(\mathcal{O}_{C_x^{\prime}})[-2] \simeq \mathcal{B}_1$ and $\Psi(\mathcal{O}_{C_x^{\prime}}(H))[-2] \simeq \mathcal{B}_0(h)$, there is an exact triangle
\begin{equation}\label{idea}
 \mathcal{U}_1 \to \mathcal{U}_0(h) \to \mathcal{H}^0(Q_x).
\end{equation}
By the exact triangle (\ref{idea}), we can obtain
\begin{align*}
v^B(\mathcal{H}^0(Q_x))&=v^B(\mathcal{U}_0(h))-v^B(\mathcal{U}_1)\\
                              &=\left(0,h,\frac{1}{2}sh+\frac{3}{2}\right),
\end{align*}
\begin{align*}
v^B((P_x)_{\mathrm{tor}})&=v^B(\mathcal{H}^{-1}(Q_x))\\
                               &=v^B(\mathcal{H}^0(Q_x))-(0,0,2)\\
                               &=\left(0,h,\frac{1}{2}sh-\frac{1}{2}\right).
\end{align*}
So we prove that $\Psi(\mathcal{O}_{C_x^{\prime}})[-2] \simeq \mathcal{B}_1$ and $\Psi(\mathcal{O}_{C_x^{\prime}}(H))[-2] \simeq \mathcal{B}_0(h)$.

First, we prove that $\Psi(\mathcal{O}_{C_x^{\prime}})[-2] \simeq \mathcal{B}_1$.

By Lemma \ref{reducelem} and the exact triangle (\ref{koko1}),  we have $\Psi(\mathcal{O}_D)=0$.

Applying the functor $\Psi$ to the sequence (\ref{koko2}), we have the exact triangle:
\[ \Psi(\mathcal{O}_D(-H)) \to \Psi(\mathcal{O}_D) \to \Psi(\mathcal{O}_{C_x^{\prime}}) \] 
Since $\Psi(\mathcal{O}_D)=0$, we have $\Psi(\mathcal{O}_{C_x^{\prime}}) \simeq \Psi(\mathcal{O}_D(-H))[1]$.

Applying $\Psi(- \otimes \mathcal{O}_{\tilde{X}}(-H))$ to the sequence (\ref{koko1}), we have 
\[ \Psi(\mathcal{O}_{\tilde{X}}(h-2H)) \to \Psi(\mathcal{O}_{\tilde{X}}(-H)) \to \Psi(\mathcal{O}_D(-H)). \]
By Lemma \ref{reducelem}, we can get isomorphisms
\begin{align*}
\Psi(\mathcal{O}_{C_x^{\prime}}) &\simeq \Psi(\mathcal{O}_D(-H))[1]\\
                                          &\simeq \Psi(\mathcal{O}_{\tilde{X}}(h-2H))[2]\\
                                          &\simeq \mathcal{B}_1[2].
\end{align*}

Next, we prove that $\Psi(\mathcal{O}_{C_x^{\prime}}(H))[-2] \simeq \mathcal{B}_0(h)$.
Applying $\Psi(- \otimes \mathcal{O}_{\tilde{X}}(H))$ to the seqence (\ref{koko2}), we have 
\[ \Psi(\mathcal{O}_D) \to \Psi(\mathcal{O}_D(H)) \to  \Psi(\mathcal{O}_{C_x^{\prime}} (H)). \]
Since $\Psi(\mathcal{O}_D)=0$, we have $ \Psi(\mathcal{O}_{C_x^{\prime}} (H)) \simeq \Psi(\mathcal{O}_D(H))$.
Applying $\Psi(- \otimes \mathcal{O}_{\tilde{X}}(H))$ to the sequence (\ref{koko1}), we have
\[ \Psi(\mathcal{O}_{\tilde{X}}) \to \Psi(\mathcal{O}_{\tilde{X}}(H)) \to \Psi(\mathcal{O}_D(H)). \]
By Lemma \ref{reducelem}, we can get isomorphisms
\begin{align*}
\Psi(\mathcal{O}_{C_x^{\prime}} (H)) &\simeq \Psi(\mathcal{O}_D(H))\\
                                               &\simeq \Psi(\mathcal{O}_{\tilde{X}}(H)\\
                                               &\simeq \mathcal{B}_0(h)[2].
\end{align*}
\end{proof}

Since Mukai vectors are integral, we obtain the following remark.
\begin{rem}\label{int}
We have 
\[\frac{1}{2}sh-\frac{1}{2} \in \mathbb{Z}.\]
\end{rem}

\begin{prop}
Let $x \in P$ be a point. Then we have $P_x/(P_x)_{\mathrm{tor}} \simeq \mathcal{U}_1$.
\end{prop}
\begin{proof}
Note that  $P_x/(P_x)_{\mathrm{tor}}$ and $\mathcal{U}_1$ are $\mu^B$-stable rank 2 torsion free sheaves.
We can calculate the Mukai vector of $P_x/(P_x)_{\mathrm{tor}}$ as follows:
\begin{align*}
v^B(P_x/(P_x)_{\mathrm{tor}})&=v^B(P_x)-v^B((P_x)_{\mathrm{tor}})\\
                                     &=\left(2,s+h,t+\frac{1}{2}sh+\frac{1}{2}\right)\\
                                     &=v^B(\mathcal{U}_1).
\end{align*}
So $P_x/(P_x)_{\mathrm{tor}}$ and $\mathcal{U}_1$ are spherical. Hence, we have $P_x/(P_x)_{\mathrm{tor}} \simeq \mathcal{U}_1$.
\end{proof}
In the next section, we construct stability conditions such that  $P_x$ is stable for any points $x \in X$.

\section{CONSTRUCTION OF STABILITY CONDITIONS}
We use the same notation as in the previous section. In this section, we prove Proposition \ref{main} (a).

\begin{dfn}\label{defB}
Let 
\[ \tilde{B}:=\frac{1}{2}l+\frac{1}{4}h+B=\frac{1}{2}s+\frac{1}{4}h \in H^2(S,\mathbb{Q}).\]
Here $l:=s-2B \in \mathrm{Pic}S$.
For $\lambda>1/2$, we can define the stability condition $\sigma_{\lambda}=(Z_{\lambda},\mathcal{C}):=(Z_{\tilde{B},\lambda h},\mathcal{C})$ on $D^b(S,\alpha)$\rm{(Example \ref{stab}}).
\end{dfn}
For simplicity, we denote the central charge $Z_{\lambda}$ by $Z$. To prove Proposition \ref{main} (a), it is sufficient to prove the following proposition.
\begin{prop}\label{main6}
Assume that $X$ is very generic. If $\sqrt{\frac{3}{8}}<\lambda<\frac{3}{4}$, then $\sigma_{\lambda}$ is generic with respect to $v$ and $P_x$ is $\sigma_{\lambda}$-stable for all $x \in X$.
\end{prop}
Before giving the proof, we need some calculations.

\begin{lem}\label{central}
For a point $x \in X$, we have
\[ Z(P_x)=2\lambda^2+\frac{3}{8}+3\lambda i. \]
For a point $x \in P$, we have
\[ Z((P_x)_{\mathrm{tor}})=1+2\lambda i . \]
\end{lem}
\begin{proof}
Note that $\mathrm{Re}(e^{\tilde{B}+\lambda ih})=(1,\tilde{B},\tilde{B}^2/2-\lambda^2)$ and $\mathrm{Im}(e^{\tilde{B}+\lambda ih})=(0,\lambda h,\lambda\tilde{B}h)$.
By (\ref{mukai}) and (\ref{mukaitor}), we can calculate values of the central charge $Z$ directly.
\end{proof}

\begin{rem}\label{phase}
Let $x \in P$ be a point. Then $\phi((P_x)_{\mathrm{tor}})<\phi(P_x)$ if and only if $\lambda<3/4$.
\end{rem}
\begin{proof}
By Lemma \ref{central}, the inequality
$\phi((P_x)_{\mathrm{tor}})<\phi(P_x)$ is equivalent to 
\[\mathrm{Re}Z((P_x)_{\mathrm{tor}})>\frac{2}{3}\mathrm{Re}Z(P_x).\]
 Solving this inequality, we have  $\lambda <3/4$.

\end{proof}

\begin{lem}
We have $\mathrm{Im}Z(E) \in \lambda \mathbb{Z} \subset \mathbb{R}$ for all $E \in K(S,\alpha)$.
\end{lem}
\begin{proof}
Let $E \in K(S,\alpha)$ and $v^B(E)=(r,c,d)$. 
By Definition \ref{defB} and Remark \ref{int}, we can calculate as follows:
\begin{align*}
\mathrm{Im}Z(E)&=\langle (r,c,d),(0,\lambda h,\lambda \tilde{B}h)\rangle\\
                     &=\lambda\Bigl\{ch-\left(\frac{1}{2}sh+\frac{1}{2}\right)\Bigr\} \in \lambda\mathbb{Z} \subset \mathbb{R}.
\end{align*}
\end{proof}

\begin{lem}\label{1/4}
Assume taht $\alpha \neq 1$. Then the following holds for all $E \in K(S,\alpha)$.
\[ \mathrm{Re}Z(E) \in \dfrac{1}{4}\mathbb{Z}+\dfrac{\mathrm{rk}E}{2}\left(2\lambda^2+\dfrac{3}{8}\right)\ \subset \mathbb{R}. \] 
\end{lem}
\begin{proof}
Let $E \in K(S,\alpha)$ and $v^B(E)=(r,c,d)$. Since $\alpha \neq 1$, the integer $r$ is even.

Note that 
\[ c\tilde{B}=\dfrac{1}{2}cs+\dfrac{1}{4}ch \in \dfrac{1}{4}\mathbb{Z} \] and
\[ \tilde{B}^2=\dfrac{1}{4}s^2+\dfrac{1}{4}sh+\dfrac{1}{8} \in \dfrac{1}{4}\mathbb{Z}+\dfrac{1}{8} \subset \mathbb{R}. \]

So we have
\begin{align*}
\mathrm{Re}Z(E)&=\langle (r,c,d),(1,\tilde{B},(\tilde{B}^2-2\lambda^2)/2) \rangle\\
                     &= c \tilde{B}-\frac{r}{2}(\tilde{B}^2-2\lambda^2)-d\\
                     &=c\tilde{B}-\frac{r}{2}\left(\tilde{B}^2-\frac{1}{2}+\frac{1}{2}-2\lambda^2\right)-d\\
                     &=c\tilde{B}-\frac{r}{2}\left(\tilde{B}^2-\frac{1}{8}\right)-\frac{r}{2}\cdot\frac{1}{2}-d+\frac{r}{2}\left(2\lambda^2+\frac{3}{8}\right) \in \frac{1}{4}\mathbb{Z}+\frac{r}{2}\left(2\lambda^2+\frac{3}{8}\right).
\end{align*}
\end{proof}

\begin{lem}
Assume that $\alpha \neq 1$.
Then we have $P_x \in \mathcal{C}$ for all $x \in X$.
\end{lem}
\begin{proof}
Take $x \in X \setminus P$.
Since $\alpha \neq 1$ and $P_x$ is rank 2 torsion free, $P_x$ is $\mu^B$-stable. 
Due to $\mathrm{Im}Z(P_x)=3\lambda >0$, we have $\mu^B(P_x) > \tilde{B}h$. Hence, we obtain $P_x \in \mathcal{T}$.

Take $x \in P$. Since $\alpha \neq 1$, $P_x/(P_x)_{\mathrm{tor}}$ is $\mu^B$-stable. 
Since $\mathrm{Im}Z(P_x/(P_x)_{\mathrm{tor}})=\lambda >0$, we have $\mu^B(P_x/(P_x)_{\mathrm{tor}})>\tilde{B}h$.
Hence, we obtain $P_x \in \mathcal{T}.$ 
\end{proof}

\begin{lem}
Assume that $\alpha \neq 1$. Let $x \in X$ be a point and $0 \neq F \subset P_x$ be a subobject in $\mathcal{C}$. Then $F \in \mathcal{T}$ and $\mathrm{Im}Z(F)>0$.
\end{lem}
\begin{proof}
Since $P_x \in \mathcal{T}$, we have $\mathcal{H}^{-1}(F)=0$. So we obtain $F \in \mathcal{T}$.

We prove that  $\mathrm{Im}Z(F)>0$. If $F$ is not torsion, $\mathrm{Im}Z(F)>0$ holds.  So we assume that $F$ is torsion.

Consider the exact sequence
\[ 0 \to F \to P_x \to \mathrm{Coker}(F \to P_x) \to 0 \]
in $\mathcal{C}$.
Taking the long exact sequence, we have the exact sequence
\[ 0 \to \mathcal{H}^{-1}(\mathrm{Coker}(F \to P_x)) \to F \to P_x \to \mathcal{H}^0(\mathrm{Coker}(F \to P_x)) \to 0\]
in $\mathrm{Coh}(S,\alpha)$. 

Since $\mathcal{H}^{-1}(\mathrm{Coker}(F \to P_x)) \in \mathcal{F}$ and $F$ is torsion, we have 
\[ \mathcal{H}^{-1}(\mathrm{Coker}(F \to P_x))=0. \]
So $F \subset (P_x)_{\mathrm{tor}}$ in $\mathrm{Coh}(S,\alpha)$. 
If $x \in X \setminus P$, we have $(P_x)_{\mathrm{tor}}=0$ and $F \neq 0$. This is contradiction.  So we can  assume that $x \in P$.
Since $(P_x)_{\mathrm{tor}}$ is an 1-dimensional pure torsion sheaf, $F$ is an 1-dimensional torsion sheaf. Hence,  $\mathrm{Im}Z(F)>0$ holds.
\end{proof}

\begin{lem}[\cite{Tod13}, Lemma 3.7]\label{toda}
Let $I_P \subset \mathcal{O}_X$ be the ideal sheaf of the plane $P$ in $X$. Then we have
\[ \mathbf{R}\sigma_{*}\Phi(\mathcal{B}_1) \simeq I_P \oplus \mathcal{O}_X(-H)^{\oplus3}.\]
\end{lem}

\begin{lem}\label{vanish}
We have $\mathrm{Hom}(\mathcal{U}_1,P_x)=0$ for all $x \in X$.
\end{lem}
\begin{proof}
There are following isomorphisms and an inclusion:
\begin{align*}
\mathrm{Hom}(\mathcal{U}_1,P_x) &\simeq \mathrm{Hom}(P_x,\mathcal{U}_1[2])\\
                                           &\simeq  \mathrm{Hom}_{\mathcal{B}_0}(\Psi(\mathbf{L}\sigma^*I_x(H))[-2],\mathcal{B}_1[2]) \\
                                          &\simeq \mathrm{Hom}_X(I_x(H)[-2],\mathbf{R}\sigma_*\Phi(\mathcal{B}_1)[2]) \\
                                          &\simeq  \mathrm{Hom}_X(I_x(H)[-2], I_P \oplus \mathcal{O}_X(-H)^{\oplus3}[2])\\
                                         &\simeq  \mathrm{Hom}_X(I_x(H)[-4], I_P \oplus \mathcal{O}_X(-H)^{\oplus3}) \\ 
                                         &\simeq \mathrm{Hom}_X(I_P \oplus \mathcal{O}_X(-H)^{\oplus3},I_x(-2H)) \\
                                          &\simeq  \mathrm{Hom}_X(I_P \oplus \mathcal{O}_X(-H)^{\oplus3},I_x(-2H)) \\
                                          &\simeq \mathrm{Hom}_X(I_P,I_x(-2H)) \oplus \mathrm{Hom}_X(\mathcal{O}_X(-H),I_x(-2H))^{\oplus3} \\
                                          &\subset  \mathrm{Hom}_X(I_P^{\vee \vee},I_x(-2H)^{\vee \vee}) \oplus \mathrm{Hom}(\mathcal{O}_X(-H)^{\vee \vee},I_x(-2H)^{\vee \vee})^{\oplus3} \\
                                         &\simeq  \mathrm{Hom}_X(\mathcal{O}_X,\mathcal{O}_X(-2H)) \oplus \mathrm{Hom}_X(\mathcal{O}_X(-H),\mathcal{O}_X(-2H))^{\oplus3}\\
&=0
\end{align*} 
The first isomorphism is given by the Serre duality for $\mathcal{A}_X$. The third isomorphism is deduced from the adjoint property. The fourth isomorphism is given by Lemma \ref{toda}. The sixth isomorphism is given by the Serre duality for $D^b(X)$.
So we have $\mathrm{Hom}(\mathcal{U}_1,P_x)=0$.
\end{proof}

\begin{lem}\label{ineq}
Let $F \in D^b(S,\alpha)$ be an object which satisfies $\mathrm{Hom}(F,F)=\mathbb{C}$. 
\begin{itemize}
\item 
Assume that $\mathrm{Im}Z(F)=\lambda$. If $\mathrm{rk}F>0$, then we have the inequality
\[ \mathrm{Re}Z(F) \ge 2 \lambda^2 - \frac{5}{8}. \] 
The equality holds if and only if the Mukai vector of $F$ is 
\[ v^B(F)=\left(2,s+h,\frac{1}{2}sh+\frac{1}{2}\right) \]

If $\mathrm{rk}F \ge 4$, then we have the inequality:
\[ \mathrm{Re}Z(F) \ge 4 \lambda^2 - \frac{5}{16}. \]
\item
Assume that $\mathrm{Im}Z(F)=2\lambda$. If $\mathrm{rk}F>0$, then we have the inequality
\[ \mathrm{Re}Z(F) > 2\lambda^2-1 .\]

If $\mathrm{rk}F \ge 4$, then we have the inequality
\[ \mathrm{Re}Z(F) \ge 4\lambda^2 - \frac{1}{2}. \]
\end{itemize}
\end{lem}
\begin{proof}
Let $v^B(F)=(r,c,d)$, $r>0$ and $L:=c-r\tilde{B} \in \mathrm{NS}(S)_{\mathbb{Q}}$.
Since $r>0$, the following holds:
\[ \mathrm{Re}Z(F)=\frac{1}{2r}(-\chi(F,F)+2r^2\lambda^2- L^2).\]
Since $\mathrm{Hom}(F,F)=\mathbb{C}$, we have $\chi(F,F)\le2$.

Note that $\mathrm{Im}Z(F)=\lambda Lh$.
Assume that $\mathrm{Im}Z(F)=\lambda$. Due to Hodge index theorem, we have the inequality
\[L^2 \le \frac{1}{2}. \] 
So we have 
\begin{align*}
\mathrm{Re}Z(F)&=\frac{1}{2r}(-\chi(F,F)+2r^2\lambda^2- L^2)\\
                     &\ge \frac{1}{2 \cdot 2}\left(-2+2 \cdot2^2 \lambda^2-\frac{1}{2}\right)\\
                      &=2\lambda^2-\frac{5}{8}.
\end{align*}
By the equality condition of Hodge index theorem, the equality holds  when $r=2$, $\chi(F,F)=2$ and $L=h/2$, this is, 
\[ v^B(F)=\left(2,s+h,t+\frac{1}{2}sh+\frac{1}{2}\right). \]

If $\mathrm{rk}F \ge 4$, then we have 
\begin{align*}
\mathrm{Re}Z(F)&=\frac{1}{2r}(-\chi(F,F)+2r^2 \lambda^2- L^2)\\
                     &\ge \frac{1}{2 \cdot 4}\left(-2+2 \cdot 4^2 \lambda^2 - \frac{1}{2}\right)\\
&=4 \lambda^2-\frac{5}{16}.                    
\end{align*}
Assume that $\mathrm{Im}Z(F)=2\lambda$. Due to Hodge index theorem, we have the inequality
\[L^2 \le 2. \]
So we have 
\begin{align*}
\mathrm{Re}Z(F)&=\frac{1}{2r}(-\chi(F,F)+2r^2 \lambda^2- L^2)\\
                     &\ge\frac{1}{2 \cdot 2}(-2 + 2 \cdot 2^2\lambda^2- 2)\\
&=2\lambda^2-1.
\end{align*}
Note that the equality holds when $r=2$, $\chi(F,F)=2$ and $L=h$.
If the equality holds, then we have $c=s+\frac{3}{2}h \notin H^2(S,\mathbb{Z})$. This is contradiction. So The equality does not hold.

If $\mathrm{rk}F \ge 4$, then we have
\begin{align*}
\mathrm{Re}Z(F)&=\frac{1}{2r}(-\chi(F,F)+2r^2 \lambda^2- L^2)\\
                     &\ge\frac{1}{2 \cdot 4}(-2+2 \cdot 4^2\lambda^2 -2)\\
&=4\lambda^2 -\frac{1}{2}.
\end{align*}
\end{proof}

\begin{prop}
Assume that $\sqrt{\frac{3}{8}}<\lambda$ and $X$ is very general. Let $E \in \mathcal{C}$ be a $\sigma_\lambda$-semistable object with Mukai vector $v$. Then we have $E \in \mathcal{T}$.
\end{prop}
\begin{proof}
Consider the natural exact sequence
\[ 0 \to \mathcal{H}^{-1}(E)[1] \to E \to \mathcal{H}^0(E) \to 0\]
in $\mathcal{C}$.

Suppose that $\mathcal{H}^{-1}(E) \neq 0$. Then $\mathrm{rk}\mathcal{H}^0(E)>0$ holds. Taking Harder-Narasimhan filtration and Jordan-H\"{o}lder filtration with respect to $\mu^B$-stability, we obtain a $\mu^B$-stable subsheaf $F \subset \mathcal{H}^{-1}(E)$.
So we obtain an exact sequence
\[0 \to F[1] \to E \to G \to 0\]
in $\mathcal{C}$.
Taking the long exact sequence, we have the exact sequence
\begin{align*}
 0 &\to F \to \mathcal{H}^{-1}(E) \to \mathcal{H}^{-1}(G)\\
 &\to 0 \to \mathcal{H}^0(E) \to \mathcal{H}^0(G) \to 0.
\end{align*}
Since $E$ is $\sigma_\lambda$-semistable, we have $\mathrm{Im}Z(F[1])>0$. So we obtain $\mathrm{Im}Z(F[1])=\lambda$ or $2\lambda$ or $3\lambda$.
Since $\mathcal{H}^{-1}(E)$ is torsion free, we have $\mathrm{rk}\mathcal{H}^0(E)>0$.

Suppose that $\mathrm{Im}Z(F[1])=3\lambda$. Then we have 
\[\mathrm{Im}Z(\mathcal{H}^0(G))+\mathrm{Im}Z(\mathcal{H}^{-1}(G)[1])=\mathrm{Im}Z(G)=0.\]
So we can deduce $\mathrm{Im}Z(\mathcal{H}^0(E))=\mathrm{Im}Z(\mathcal{H}^0(G))=0$. Since $\mathrm{rk}\mathcal{H}^0(E)>0$, $\mathrm{Im}Z(\mathcal{H}^0(E))$ must be positive. This is contradiction.
Therefore, $\mathrm{Im}Z(F[1])=\lambda$ or $2\lambda$.

Assume that $\mathrm{Im}Z(F[1])=\lambda$. By Lemma \ref{ineq}, we have
\[ \mathrm{Re}Z(F) \ge 2\lambda^2-\frac{5}{8}.\]
This implies
\begin{equation}\label{huto1}
\mathrm{Re}Z(F[1]) \le -2\lambda^2+\frac{5}{8}.
\end{equation}
Since $E$ is $\sigma_\lambda$-semistable, we have $\phi(F[1]) \le \phi(E)$. So we obtain 
\begin{equation}\label{huto2}
 \mathrm{Re}Z(F[1]) \ge \frac{1}{3}\mathrm{Re}Z(E)=\frac{2}{3}\lambda^2+\frac{1}{8}.
\end{equation}
By the inequalities (\ref{huto1}) and (\ref{huto2}), we can obtain the  inequality
\begin{equation}\label{huto3}
 -\lambda^2+\frac{5}{8} \ge \frac{2}{3}\lambda^2+\frac{1}{8}.
\end{equation}
Solving the inequality (\ref{huto3}), we have $\sqrt{3}/{4} \ge \lambda$. This is contradiction.

Assume that $\mathrm{Im}Z(F[1])=2\lambda$. By Lemma \ref{ineq}, we have 
\[ \mathrm{Re}Z(F) \ge 2\lambda^2 -1. \]
This implies
\begin{equation}\label{huto4}
 \mathrm{Re}Z(F[1]) \le -2\lambda^2+1. 
\end{equation}
Since $E$ is $\sigma_\lambda$-semistable, we have $\phi(F[1]) \le \phi(E)$.
So we obtain 
\begin{equation}\label{huto5}
 \mathrm{Re}(F[1]) \ge \frac{2}{3}\mathrm{Re}Z(E)=\frac{4}{3}\lambda^2+\frac{1}{4}. 
\end{equation}
By the inequalities (\ref{huto4}) and (\ref{huto5}), we can obtain the inequality
\begin{equation}\label{huto6}
 -2\lambda^2+1 \ge \frac{4}{3}\lambda^2+\frac{1}{4}.
\end{equation}
Solving the inequality (\ref{huto6}), we have ${3}\sqrt{10}/{20} \ge \lambda$. This is contradiction.
\end{proof}

From now on, we prove Proposition \ref{main6}.
First, we prove the generality of the stability conditions.

\begin{lem}\label{generic}
Assume that $\sqrt{\frac{3}{8}}<\lambda<\frac{3}{4}$ and $X$ is very general. Then $\sigma_\lambda$ is generic with respect to the Mukai vector $v$.
\end{lem}
\begin{proof}
It is sufficient to prove that $\sigma_\lambda$-semistable objects with Mukai vector $v$ are $\sigma_\lambda$-stable.
Let $E \in \mathcal{C}$ be a $\sigma_\lambda$-semistable object with Mukai vector $v$.  Suppose that $E$ is not $\sigma_\lambda$-stable. Then there is a exact sequence
\[ 0 \to F \to E \to G \to 0\]
in $\mathcal{C}$ such that $\phi(E)=\phi(F)$.
Now $F$ is also $\sigma_\lambda$-semistable. Taking Jordan-H\"{o}lder filtration of $F$, we can assume that $F$ is $\sigma_\lambda$-stable. Since $\mathrm{Im}Z(E)=3\lambda$ and $\phi(F)=\phi(E)$, we have $\mathrm{Im}Z(F)=\lambda$ or $2\lambda$.

Suppose that $\mathrm{Im}Z(F)=\lambda$. Since $\phi(E)=\phi(F)$, we have 
\[ \mathrm{Re}Z(F)=\dfrac{1}{3}\left(2\lambda^2+\dfrac{3}{8}\right). \]
Assume that $\mathrm{rk}F=0$. Due to Lemma \ref{1/4}, we have 
\[\mathrm{Re}Z(F) \in \frac{1}{4}\mathbb{Z}.\]
By $1/2<\lambda<3/4$, we have
\[\dfrac{1}{4}<\mathrm{Re}Z(F)=\dfrac{1}{3}\left(2\lambda^2+\dfrac{3}{8}\right)<\dfrac{1}{2}.\]
This is contradiction. Assume that $\mathrm{rk}F=2$. Due to Lemma \ref{1/4}, we have 
\[\mathrm{Re}Z(F)-\left(2\lambda^2+\dfrac{3}{8}\right)=-\dfrac{2}{3}\left(2\lambda^2+\dfrac{3}{8}\right) \in \dfrac{1}{4}\mathbb{Z}.\]
Since $\sqrt{\frac{3}{8}}<\lambda<\frac{3}{4}$, we have the inequarity
\[ -\frac{3}{4}<-\frac{2}{3}\left(2\lambda^2+\dfrac{3}{8}\right)<-\frac{7}{12}. \]
This is contradiction. Hence, we have $\mathrm{rk}F \ge 4$. By Lemma \ref{ineq}, we obtain the inequality
\begin{equation}\label{hutoa}
 \frac{1}{3}\left(2\lambda^2+\dfrac{3}{8}\right)=\mathrm{Re}Z(F) \ge 4\lambda^2-\frac{5}{16}.
\end{equation}
Solving the inequality (\ref{hutoa}), we have the  inequality
\[ \lambda^2 \le \frac{21}{160}. \]
By the assumption $\lambda^2 > 3/8$, this is contradiction.

Suppose that $\mathrm{Im}Z(F)=2\lambda$. Since $\phi(E)=\phi(F)$, we have
\[ \mathrm{Re}Z(F)=\dfrac{2}{3}\left(2\lambda^2+\dfrac{3}{8}\right).\]
Assume that $\mathrm{rk}F=0$. Due to Lemma \ref{1/4}, we have 
\[\mathrm{Re}Z(F) \in \dfrac{1}{4}\mathbb{Z}.\]
By $\sqrt{\frac{3}{8}} < \lambda <\frac{3}{4}$, we have the inequality
\[\frac{3}{4} <\mathrm{Re}(F)=\frac{2}{3}\left(2\lambda^2+\frac{3}{8}\right) <1. \]
This is contradiction. Assume that $\mathrm{rk}F=2$. Due to Lemma \ref{1/4}, we have
\[\mathrm{Re}Z(F)-\left(2\lambda^2+\dfrac{3}{8}\right)=-\dfrac{1}{3}\left(2\lambda^2+\dfrac{3}{8}\right) \in \dfrac{1}{4}\mathbb{Z}.\]
By $\sqrt{\frac{3}{8}}<\lambda<\frac{3}{4}$, we have the inequality
\[ -\frac{1}{2} <-\frac{1}{3}\left(2\lambda^2+\frac{3}{8}\right) <-\frac{3}{8}. \]
This is contradiction. Hence, we have $\mathrm{rk}F \ge 4$. By Lemma \ref{ineq}, we obtain the inequality
\begin{equation}\label{hutob}
 \frac{2}{3}\left(2 \lambda^2+\frac{3}{8}\right)=\mathrm{Re}Z(F) \ge 4\lambda^2-\frac{1}{2}. 
\end{equation}
Solving the inequality (\ref{hutob}), we have the  inequality
\[ \frac{9}{32} \ge \lambda^2. \]
By the assumption $3/8 < \lambda^2$, this is contradiction.
\end{proof}
Finaly, we prove the stability of $P_x$ for all $x \in X$.

\begin{proof}[Proof of Proposition \ref{main6}]
Take $x \in X$. By Lemma \ref{generic}, it is suffiient to prove that $P_x$ is $\sigma_\lambda$-semistable. 
Suppose that $P_x$ is not $\sigma_\lambda$-semistable. Then there is an exact sequence
\[ 0 \to F \to P_x \to G \to 0. \]
in $\mathcal{C}$ such that $ \phi(F)>\phi(P_x)$. Taking Harder-Narasimhan filtration and Jordan-H\"{o}lder filtration of $F$, we can assume that $F$ is $\sigma_\lambda$-stable. Since $P_x \in \mathcal{T}$, the object $F$ is also contained in $\mathcal{T}$.

First, we prove that $\mathrm{rk}F>0$. Assume that $\mathrm{rk}F=0$. Since $\mathcal{H}^{-1}(G) \in \mathcal{F}$, we have $\mathcal{H}^{-1}(G)=0$.  So $F$ is a subsheaf of $P_x$.  Since $P_x$ is torsion free for $x \in X \setminus P$, it is sufficient to consider the case of $x \in P$. Now $F \subset (P_x)_{\mathrm{tor}}$ and $(P_x)_{\mathrm{tor}}$ is an 1 -dimensional pure torsion sheaf.  So we can write $v^B(F)=(0,h,k)$ for some $ k \in \mathbb{Z}$. Since $(P_x)_{\mathrm{tor}}/F$ is a zero dimensional torsion sheaf, we have $\phi((P_x)_{\mathrm{tor}}) \ge \phi(F)$.
By Remark \ref{phase}, we obtain the inequality
\[ \phi((P_x)_{\mathrm{tor}}) > \phi(F) > \phi((P_x)_{\mathrm{tor}}). \] 
This is contradiction. Therefore, we have $\mathrm{rk}F > 0$. Hence, we have $\mathrm{Im}Z(F)=\lambda$ or $2\lambda$.  Assume that $\mathrm{Im}Z(F)=\lambda$.
By $\phi(F) > \phi(P_x)$, we have the inequality
\[ \mathrm{Re}Z(F) < \frac{1}{3}\mathrm{Re}Z(P_x)=\frac{1}{3}\left(2\lambda^2+\frac{3}{8}\right). \]
We will prove that $\mathrm{rk}F=2$. So we assume that $\mathrm{rk}F \ge 4$. By Lemma \ref{ineq}, we have the inequality
\[ \frac{1}{3}\left(2\lambda^2+\frac{3}{8}\right)>\mathrm{Re}Z(F) \ge 4\lambda^2-\frac{5}{16}. \]
Solving the  inequality
\[\frac{1}{3}\left(2\lambda^2+\frac{3}{8}\right)>4\lambda^2-\frac{5}{16},\]
we have 
\[ \frac{21}{160} \ge \lambda^2. \]
This is contradiction. Hence, $\mathrm{rk}F=2$. Due to Lemma \ref{1/4}, we have 
\[\mathrm{Re}Z(F)-\left(2\lambda^2+\dfrac{3}{8}\right) \in \dfrac{1}{4}\mathbb{Z}.\]
By Lemma \ref{ineq}, we have the inequality
\[ \frac{1}{3}\left(2\lambda^2+\frac{3}{8}\right)>\mathrm{Re}Z(F) \ge \lambda^2 -\frac{5}{8}. \]

This implies
\[-\frac{2}{3}\left(2\lambda^2+\frac{3}{8}\right)>\mathrm{Re}Z(F)-\left(2\lambda^2+\dfrac{3}{8}\right) \ge -1. \]
By the assumption  $\sqrt{\frac{3}{8}}<\lambda$, we have the inequality
\[-\frac{3}{4}>\mathrm{Re}Z(F)-\left(2\lambda^2+\dfrac{3}{8}\right) \ge -1. \]
So we obtain $\mathrm{Re}Z(F)=2\lambda^2-5/8$. By Lemma \ref{ineq}, we have 
\begin{align*}
 v^B(F)&=(2,s+h,t+\dfrac{1}{2}sh+\dfrac{1}{2})\\
          &=v^B(\mathcal{U}_1). 
\end{align*}
Since $\mathrm{Pic}S=\mathbb{Z}h$, the sperical sheaf $F$ is torsion free. So we have  $F \simeq \mathcal{U}_1$.  By Lemma \ref{vanish}, the morphism $F \hookrightarrow P_x$ is zero. Hence, we have $\mathcal{U}_1=0$. This is contradiction.
Assume that $\mathrm{Im}Z(F)=2\lambda$. If $\mathrm{rk}F=2$, then $\mathrm{rk}G=0$. Since $\mathrm{Pic}S=\mathbb{Z}h$, we have $\mathrm{Im}Z(G) \in 2\mathbb{Z} \lambda \subset \mathbb{R}$. However, $\mathrm{Im}Z(G)=\lambda$ holds. This is contradiction.  So we  prove that $\mathrm{rk}F=2$.
Assume that $\mathrm{rk}F \ge 4$. Since $\phi(F)<\phi(P_x)$, we have the inequality
\[ \mathrm{Re}Z(F) <\frac{2}{3}\left(2 \lambda^2+\frac{3}{8}\right). \]
By Lemma \ref{ineq}, we have the inequality
\[ \mathrm{Re}Z(F) \ge 4\lambda^2-\frac{1}{2}. \]
So we obtain
\begin{equation}\label{ht}
 \frac{2}{3}(2\lambda^2+\frac{3}{8}) > 4\lambda^2-\frac{1}{2}. 
\end{equation}
Solving the inequality (\ref{ht}), we have the inequality:
\[ \frac{9}{32} >\lambda^2. \]
This gives contradiction.  So $P_x$ is $\sigma_{\lambda}$-stable.
\end{proof}

\address{Graduate School of Mathematical Sciences, University of Tokyo, Meguro-ku, Tokyo 153-8914, Japan}

\it{E-mail-address}\rm{:} genki.oouchi@ipmu.jp
\end{document}